\documentclass[11pt, reqno]{article}
\usepackage{amsmath,amssymb,amsfonts,amsthm}
\usepackage{color}

\usepackage[colorlinks=true,linkcolor=blue,citecolor=blue,urlcolor=blue,pdfborder={0 0 0}]{hyperref}
\usepackage{cleveref}

\usepackage{caption}
\usepackage{thmtools}

\usepackage{mathrsfs}

\crefname{theorem}{Theorem}{Theorems}
\crefname{thm}{Theorem}{Theorems}
\crefname{lemma}{Lemma}{Lemmas}
\crefname{lem}{Lemma}{Lemmas}
\crefname{remark}{Remark}{Remarks}
\crefname{prop}{Proposition}{Propositions}
\crefname{defn}{Definition}{Definitions}
\crefname{corollary}{Corollary}{Corollaries}
\crefname{conjecture}{Conjecture}{Conjectures}
\crefname{question}{Question}{Questions}
\crefname{chapter}{Chapter}{Chapters}
\crefname{section}{Section}{Sections}
\crefname{figure}{Figure}{Figures}
\newcommand{\crefrangeconjunction}{--}

\theoremstyle{plain}
\newtheorem{thm}{Theorem}[section]

\newtheorem{lemma}[thm]{Lemma}
\newtheorem{theorem}[thm]{Theorem}

\newtheorem{corollary}[thm]{Corollary}

\newtheorem{prop}[thm]{Proposition}

\theoremstyle{definition}

\newtheorem{definition}[thm]{Definition}

\theoremstyle{remark}

\numberwithin{equation}{section}

\newcommand{\D}{\mathbb D}
\renewcommand{\P}{\mathbb P}
\newcommand{\E}{\mathbb E}
\newcommand{\C}{\mathbb C}
\newcommand{\R}{\mathbb R}

\newcommand{\cD}{\mathcal D}
\newcommand{\cE}{\mathcal E}

\newcommand{\sE}{\mathscr E}

\usepackage[margin=1.1in]{geometry}
\usepackage{extarrows}
\usepackage{titling}
\usepackage{commath}
\usepackage{mathtools}
\usepackage{titlesec}
\newcommand{\eps}{\varepsilon}
\usepackage{bbm}
\usepackage{setspace}
\usepackage{enumitem}
\usepackage{tikz}

\usepackage{cite}

\newcommand{\bP}{\mathbf P}
\newcommand{\bE}{\mathbf E}

\renewcommand{\emptyset}{\varnothing}

\def\C{{\mathcal C}}

\newcommand{\tn}{|\kern-.1em|\kern-0.1em|}

\newcommand\be{\begin{equation}}
\newcommand\ee{\end{equation}}

\def\eps{\varepsilon}

\pretitle{\begin{flushleft}\Large}
\posttitle{\par\end{flushleft}\vskip 0.5em}
\preauthor{\begin{flushleft}}
\postauthor{
\\
\vspace{0.5em}
\footnotesize{Statslab, DPMMS, University of Cambridge.\\ Email: 
\href{mailto:t.hutchcroft@maths.cam.ac.uk}{t.hutchcroft@maths.cam.ac.uk}}
\end{flushleft}}
\predate{\begin{flushleft}}
\postdate{\par\end{flushleft}}
\title{\bf Harmonic Dirichlet functions on planar graphs}

\renewenvironment{abstract}
 {\par\noindent\textbf{\abstractname.}\ \ignorespaces}
 {\par\medskip}

\titleformat{\section}
  {\normalfont\large \bf}{\thesection}{0.4em}{}
\titleformat{\subsection}
  {\normalfont\large \bf}{\thesubsection}{0.4em}{}

\author{{\bf Tom Hutchcroft}}

\newcommand{\BD}{\mathbf{BD}}
\newcommand{\HD}{\mathbf{HD}}
\newcommand{\BH}{\mathbf{BH}}
\newcommand{\BHD}{\mathbf{BHD}}

\renewcommand{\Cap}{\mathrm{Cap}}

\newcommand{\bM}{\mathbf{M}}
\newcommand{\bD}{\mathbf{D}}
\newcommand{\Cont}{\mathsf{Cont}}
\newcommand{\Disc}{\mathsf{Disc}}
\renewcommand{\C}{\mathbb{C}}

\begin{document}

\date{\small{\today}}

\maketitle

\begin{abstract}
Benjamini and Schramm (\emph{Invent.\ Math.}, 126(3):565-587, 1996) used circle packing to prove that every transient, bounded degree planar graph admits non-constant harmonic functions of finite Dirichlet energy. We refine their result, showing in particular that for every transient, bounded degree, simple planar triangulation $T$ and every circle packing of $T$ in a domain $D$, there is a canonical, explicit bounded linear isomorphism between the space of harmonic Dirichlet functions on $T$ and the space of harmonic Dirichlet functions on $D$. 
\end{abstract}


\section{Introduction}

A \textbf{circle packing} is a collection $P$ of discs in the Riemann sphere $\C \cup \{\infty\}$ such that distinct discs in $P$ do not overlap (i.e., have disjoint interiors), but may be tangent. Given a circle packing $P$, its \textbf{tangency graph} (or \textbf{nerve}) is the graph whose vertices are the discs in $P$ and where two vertices are connected by an edge if and only if their corresponding discs are tangent. The Circle Packing Theorem \cite{K36,Th78} states that every finite, simple\footnote{A graph is said to be \textbf{simple} if it does not contain any loops or multiple edges.} planar graph may be represented as the tangency graph of a circle packing, and that if the graph is a triangulation (i.e., every face has three sides) then the circle packing is unique up to M\"obius transformations and reflections. See e.g.\ \cite{St05,Rohde11} for further background on circle packing. 

The Circle Packing Theorem was extended to infinite, simple planar triangulations by He and Schramm \cite{HS93,HeSc,Schramm91,he1999rigidity}. In particular, they showed that if the triangulation is \emph{simply connected}, meaning that the surface formed by gluing triangles according to the combinatorics of the triangulation is homeomorphic to the plane, then the triangulation can be circle packed in either the disc or the plane, but not both\footnote{Here the word \textbf{in} is being used in a technical sense to mean that the \textbf{carrier} of the circle packing is equal to either the disc or the plane, see \cref{subsec:mapsdcp}.}; we call the triangulation \textbf{CP parabolic} or \textbf{CP hyperbolic} accordingly. More generally, they showed that, in the CP hyperbolic case, the triangulation can be circle packed in \emph{any} simply-connected domain $D \subsetneq \C$. These results can be viewed as discrete analogue of the Riemann mapping theorem and of the uniformization theorem for Riemann surfaces. Indeed, the theory of circle packing is closely related to the theory of conformal mapping and geometric function theory, see e.g.\ \cite{HS93,RS87,Rohde11,St05,MR3755822} and references therein.

He and Schramm  also pioneered the use of circle packing to study probabilistic questions about planar graphs, showing in particular that a bounded degree, simply connected, planar triangulation is CP parabolic if and only if it is recurrent for simple random walk \cite{HeSc}. This result was recently generalised by Gurel-Gurevich, Nachmias, and Suoto \cite{GGNS15}, who proved that a (not necessarily simply connected) bounded degree planar triangulation admitting a circle packing in a domain $D$ is recurrent for simple random walk if and only if the domain is recurrent for Brownian motion.


A more detailed study of the relationship between circle packing and random walks was initiated by Benjamini and Schramm \cite{BS96a},
who proved in particular that if $T$ is a bounded degree triangulation circle packed in the unit disc $\D$, then the random walk on $T$ converges almost surely to a point in the boundary $\partial \D$, and the law of this limit point is non-atomic. 
They used this to deduce the existence of various kinds of \emph{harmonic functions} on transient, bounded degree planar graphs. 
Recall that a function $h$ on the vertex set of a simple, locally finite graph $G=(V,E)$ is said to be \textbf{harmonic} if 
\[h(v)=\frac{1}{\deg(v)}\sum_{u\sim v}h(u) \]
for every $v\in V$. Here and elsewhere, we write $V$ and $E$ for the vertex and edge sets of a graph $G$, and write $u\sim v$ if the vertices $u$ and $v$ are adjacent in $G$. 
 Three particularly important and probabilistically meaningful classes of harmonic functions are the \emph{bounded harmonic functions}, 
  the \emph{positive harmonic functions},
    and the \emph{harmonic Dirichlet functions}. It is an easy consequence of the Benjamini-Schramm convergence theorem that every bounded degree, transient planar graph admits non-constant harmonic functions in each of these three classes. 
     Here, a \textbf{harmonic Dirichlet function} on a graph with oriented edge set $E^\rightarrow$ is a harmonic function $h$ such that
\[
\cE(h) = \frac{1}{2}\sum_{e\in E^\rightarrow} \left[ h\big(e^+\big)-h\big(e^-\big) \right]^2 <\infty.
\]
We denote the space of harmonic Dirichlet functions on a graph $G$ by $\HD(G)$ and the space of bounded harmonic Dirichlet functions on $G$ by $\BHD(G)$. For each vertex $v$ of $G$, $\|h\|=h(v)^2+\cE(h)$ is a norm on $\HD(G)$, and $\BHD(G)$ is dense in $\HD(G)$ with respect to this norm \cite[Theorem 3.73]{Soardibook}. (Without the $h(v)^2$ term this would be a seminorm rather than a norm.)  Harmonic Dirichlet functions and function spaces on domains are defined similarly; see \cref{subsec:Dirichlet} for details.

More recently, Angel, Barlow, Gurel-Gurevich, and Nachmias \cite{ABGN14} showed that \emph{every} bounded harmonic function and every positive harmonic function on a bounded degree, simply connected, simple  planar triangulation can be represented geometrically in terms of the triangulation's circle packing in the unit disc. 
A similar representation theorem for bounded (but not positive) harmonic functions using a different embedding, the \emph{square tiling}, was obtained slightly earlier by Georgakopoulos~\cite{G13}. Simpler proofs of both results for bounded harmonic functions have since been obtained by Peres and the author \cite{hutchcroft2015boundaries}.

In this paper we establish a similar representation theorem for harmonic Dirichlet functions.  We begin with a simple form of the result that can be stated with minimum preparation. 
We say that two functions $\phi$ and $\psi$ on the vertex set of a graph are \textbf{asymptotically equal} if the set $\{v\in V: |\phi(v)-\psi(v)|\geq \eps\}$ is finite for every $\eps>0$. 

\begin{theorem}
\label{thm:isomorphismdisc}
Let $T$ be a bounded degree, simply connected, simple, planar triangulation, let $P$ be a circle packing of $T$ in the unit disc $\D$, and let $z:V\to \D$ be the function sending vertices to the centres of their corresponding discs. 
\begin{enumerate}
	\item For each bounded harmonic Dirichlet function $h \in \BHD(T)$, there exists a unique harmonic Dirichlet function $H \in \HD(\D)$ such that $h$ and $H\circ z$ are asymptotically equal. 
	\item For each bounded harmonic Dirichlet function $H \in \BHD(\D)$, there exists a unique harmonic Dirichlet function $h \in \HD(T)$ such that $h$ and $H \circ z$ are asymptotically equal. 
\end{enumerate}
Moreover, the function assigning each $h\in \BHD(T)$ to the unique $H \in \HD(\D)$ such that $H\circ z$ is asymptotically equal to $h$ can be uniquely extended to a bounded linear isomorphism from $\HD(T)$ to $\HD(\D)$. 
\end{theorem}
By a bounded linear isomorphism we mean a bounded linear map with a bounded inverse; such an isomorphism need not be an isometry. A more general form of our theorem, applying in particular to bounded degree, multiply-connected planar triangulations circle packed in arbitrary domains, is given in \cref{thm:isomorphismgeneral}. See \eqref{eq:discdef} and \eqref{eq:contdef} for an explicit description of the isomorphism.



Note that \cref{thm:isomorphismdisc,thm:isomorphismgeneral} are much stronger than those available for bounded and or positive harmonic functions. For example, the representation theorem for bounded harmonic functions \cite{ABGN14} requires one to take integrals over the harmonic measure on the boundary, which is not particularly well understood and can be singular with respect to the corresponding measure for Brownian motion. As a consequence, there can exist bounded harmonic functions $h$ on $T$ such that $h$ is not asymptotically equal to $H \circ z$ for any bounded harmonic function $H$ on $\D$. 
The difference in strength between these theorems is unsurprising given that the \emph{existence} of non-constant harmonic Dirichlet functions is known to be stable under various perturbations of the underlying space \cite{Soardi93,holopainen1997p,Gab05}, while the existence of non-constant bounded harmonic functions is known to be unstable in general under similar perturbations \cite{BS96a}. 


\subsection{Applications}

\cref{thm:isomorphismdisc} also allows us to deduce various facts about the boundary behaviour of harmonic Dirichlet functions on circle packings of triangulations from the corresponding facts about harmonic Dirichlet functions on the unit disc. For example, we immediately obtain a representation theorem for the harmonic Dirichlet functions on $T$ in terms of boundary functions, similar to that obtained for bounded harmonic functions in \cite{ABGN14}. 
We say that a Borel function $\phi: \partial \D \to \R$ 
is \textbf{Douglas integrable} if
\begin{equation}
\label{eq:Douglas}
  \cD(\phi):= \frac{1}{4\pi}\int_{\partial \D} \int_{\partial \D} \left| \frac{\phi(\xi)-\phi(\zeta)}{\xi - \zeta}\right|^2 \dif \xi \dif \zeta <\infty. 
\end{equation}
Note in particular that every Lipschitz function on $\partial \D$ is Douglas integrable. 
It is a classical theorem of Douglas \cite{MR1501590} that a harmonic function $H:\D\to\R$ is Dirichlet if and only if it is the extension of a Douglas integrable function $\phi:\partial\D\to\R$, and in this case $\cD(\phi)=\cE(h)$. This equality is known as the \emph{Douglas integral formula}.  Thus, we obtain the following corollary to \cref{thm:isomorphismdisc}.

\begin{corollary}
\label{cor:douglas}
Let $T$ be a bounded degree, simply connected, simple, planar triangulation and let $P$ be a circle packing of $T$ in the unit disc $\D$. Then a function $h:V\to \R$ is a harmonic Dirichlet function if and only if there exists a Douglas integrable Borel function $\phi:\partial \D\to\R$  
such that
\[h(v) = \bE_v \left[ \phi\left(\lim_{n\to\infty} z(X_n) \right) \right] \qquad \text{for every vertex $v$.}\]
\end{corollary}

We remark that there is a generalization of the Douglas integral formula to other domains due to Doob \cite{doob1962boundary}, and that related results for graphs have been announced by Georgakopoulos and Kaimanovich \cite{georgakopoulos2015group}. The results of Doob could be combined with \cref{thm:isomorphismgeneral} to obtain versions of \cref{cor:douglas} for more general domains. We do not pursue this here. 

\medskip

Similarly, we can immediately deduce the following very strong boundary convergence result from \cref{thm:isomorphismdisc} together with a theorem of Nagel, Rudin, and Shapiro \cite{MR672838}. 

\begin{corollary}[Boundary convergence in exponentially tangential approach regions]
\label{cor:exponentiallytangential}
Let $T$ be a bounded degree, simply connected, simple, planar triangulation, let $P$ be a circle packing of $T$ in the unit disc $\D$, and let $z:V\to \D$ be the function sending vertices to the centres of their corresponding discs. Then for each $h\in \BHD(T)$, the following holds for Lebesgue-a.e.\ $\xi\in \partial \D$: For every sequence of vertices $v_1,v_2,\ldots$ of $T$ such that $z(v_i) \to \xi$ and 
\[\limsup_{i\to\infty} |z(v_i)-\xi| \log \frac{1}{1-|z(v_i)|} <\infty,\]
the limit $\lim_{i\to\infty} h(v_i)$ exists. 
\end{corollary}

See \cite{MR3185375} and references therein for several further results concerning the boundary behaviour of harmonic Dirichlet functions on the unit disc. 

Together with the Poisson boundary identification result of \cite{ABGN14}, \cref{cor:douglas} gives us a good understanding of the relationship between the space of bounded Harmonic Dirichlet functions $\BHD(T)$ and the space of all bounded harmonic functions, denoted $\BH(T)$: The latter is identified with the space of bounded Borel functions $L^\infty(\partial \D)$, while the former is identified with the space of bounded Douglas integrable functions on $\partial \D$. In particular, this allows us to easily generate many examples of bounded harmonic functions on $T$ that are not Dirichlet, such as harmonic extensions of indicator functions.  Moreover, since the identification of $\BH(T)$ and $L^\infty(\partial\D)$ is easily seen to be a homeomorphism when $\BH(T)$ is equipped with the topology of pointwise convergence and $L^\infty(\partial\D)$ is given the subspace topology from $L^1(\partial \D)$, and since the Lipschitz functions are dense in $L^1(\partial \D)$, we obtain the following interesting corollary concerning harmonic functions on triangulations. 

\begin{corollary}
\label{cor:Ori}
Let $T$ be a bounded degree, simply connected, simple, planar triangulation. Then $\BHD(T)$ is dense in $\BH(T)$ with respect to the topology of pointwise convergence.
\end{corollary}

A nice feature of this corollary is that it is an `intrinsic' result, whose statement does not make any reference to circle packing. 
\cref{cor:douglas,cor:exponentiallytangential,cor:Ori} all have straightforward extensions to simply connected, weighted, polyhedral planar with bounded codegrees and bounded local geometry, both of which follow from \cref{thm:isomorphismgeneral}.

\cref{thm:isomorphismdisc} and its generalization \cref{thm:isomorphismgeneral} are also useful in the study of uniform spanning forests of planar graphs, for which closed linear subspaces of $\HD(T)$ correspond, roughly speaking, to possible boundary conditions at infinity for the spanning forest measure. In particular, \cref{thm:isomorphismgeneral} will be applied in forthcoming work with Nachmias on uniform spanning forests of multiply-connected planar maps. 





\subsection{The Dirichlet space}
\label{subsec:Dirichlet}

We begin by reviewing the definitions of the Dirichlet spaces in both the discrete and continuous cases, as well as some of their basic properties. For further background, we refer the reader to \cite{LP:book,Soardibook} in the discrete case, and \cite{AnLyPe99} and references therein for the continuous case.

Recall that a \textbf{network} is a graph $G=(V,E)$ (which in this paper will always be locally finite and connected) together with an assignment $c:E\to(0,\infty)$ of positive \textbf{conductances} to the edges of $G$. The random walk on a locally finite network is the Markov process that, at each step, chooses an edge to traverse from among those edges emanating from its current position, where the probability of choosing a particular edge is proportional to its conductance.  
Let $G=(V,E)$ be a network, and let $E^\rightarrow$ be the set of oriented edges of $G$. The \textbf{Dirichlet energy} of a function
$\phi: V\to \R$ is defined to be
\[\sE(\phi) = \frac{1}{2}\sum_{e\in E^\rightarrow} c(e) \left(\phi\left(e^-\right)-\phi\left(e^+\right)\right)^2.\]
We say that $\phi$ is a \textbf{Dirichlet function} (or equivalently that $\phi$ has \textbf{finite energy}) if $\sE(\phi)<\infty$. 
 The space of Dirichlet functions on $G$ and the space of harmonic Dirichlet functions on $G$ are denoted by $\bD(G)$ and $\HD(G)$ respectively. These spaces are both Hilbert spaces with respect to the inner product  
 \begin{equation}
 \label{eq:innerproductdefdisc}
\langle \phi,\psi \rangle =
\phi(o)\psi(o)+\frac{1}{2}\sum_{e\in E^\rightarrow} c(e) \left[\phi\left(e^-\right)-\phi\left(e^+\right)\right]\left[\psi\left(e^-\right)-\psi\left(e^+\right)\right],
 \end{equation}
 where $o$ is a fixed root vertex. (It is easily seen that different choices of $o$ yield equivalent norms.)
We denote the space of \emph{bounded Dirichlet functions} by $\BD(G)$ and the space of \emph{bounded harmonic Dirichlet functions} by $\BHD(G)$. These spaces are dense in $\bD(G)$ and $\HD(G)$ respectively, see \cite[Page 314]{LP:book} and \cite[Theorem 3.73]{Soardibook}. 

Let $\bD_0(G)$ be the closure in $\bD(G)$ of the space of finitely supported functions. 
If $G$ is transient,
then every Dirichlet function $\phi \in \bD(G)$ has a unique decomposition 
\begin{equation}
\label{eq:Royden}
\phi = \phi_{\bD_0} + \phi_{\HD}
\end{equation}
where $\phi_{\bD_0}\in \bD_0(G)$ and $\phi_{\HD} \in \HD(G)$, known as the \textbf{Royden decomposition} of $\phi$ \cite[Theorem 3.69]{Soardibook}. In other words, $\bD(G)=\bD_0(G)\oplus \HD(G)$. (Note that this is \emph{not} necessarily an orthogonal decomposition, although $\bD_0(G)$ and $\HD(G)$ are orthogonal with respect to the Euclidean seminorm $\cE$, see \cite[Lemma 3.66]{Soardibook}.)  Let $\langle X_n \rangle_{n\geq0}$ be a random walk on $G$. It is a theorem of Ancona, Lyons, and Peres \cite{AnLyPe99}, which complements earlier results of Yamasaki \cite{yamasaki1986ideal}, that the limit
$\lim_{n\to\infty} \phi(X_n)$
exists almost surely for each $\phi \in \bD(G)$, that
\begin{equation}
\label{eq:ALPlimitdisc}
\lim_{n\to\infty} \phi(X_n) = \lim_{n\to\infty} \phi_{\HD}(X_n)
\end{equation}
almost surely, and moreover that
 $\phi_{\HD}$ can be expressed as
\begin{equation}
\label{eq:ALPlimitidentification}
\phi_{\HD}(v) = \bE_v\left[ \lim_{n\to\infty} \phi(X_n) \right],
\end{equation}
where $\bE_v$ denotes the expectation with respect to the random walk $\langle X_n \rangle_{n\geq0}$ started at $v$. See also \cite[Theorem 9.11]{LP:book}. [The referee has informed us that the almost sure existence of the limit $\lim_{n\to\infty}\phi(X_n)$ was in fact originally proven by Silverstein in 1974 \cite{MR0350876}, independently of Ancona, Lyons, and Peres.]


A similar theory holds in the continuum. If $D \subseteq \C$ is a domain, the \textbf{Dirichlet energy} of a locally $L^2$, weakly differentiable\footnote{Recall that a function or vector field $\Phi:D\to \R^d$, $d\geq 1$, is said to be \textbf{locally integrable} if $\int_A \|\Phi(z)\|\dif z<\infty$ for every precompact open subset $A$ of $D$, and \textbf{locally $L^2$} if $\int_A \|\Phi(z)\|^2\dif z<\infty$ for every precompact open subset $A$ of $D$. 
A locally integrable vector field $W:D\to \R^2$ is said to be a \textbf{weak gradient} of the locally integrable function $\Phi : D\to \R$ if the identity
$\int_D \Psi W \dif z = -\int_D \Phi \nabla \Psi \dif z$ holds for every smooth, compactly supported function $\Psi$ on $D$. We say that a locally integrable function  $\Phi:D\to \R$ is \textbf{weakly differentiable} if it admits a weak gradient. The weak gradient of a locally integrable, weakly differentiable $\Phi:D\to \R^2$ is unique up to almost-everywhere equivalence, and is denoted by $\nabla \Phi$.
The weak gradient coincides with the usual gradient of $\Phi$ at $z$ if $\Phi$ is differentiable on an open neighbourhood of $z$.
} 
function $\Phi: D \to \R$ on $D$ is defined to be
\[\sE(\Phi) = \int_D \|\nabla \Phi (z)\|^2 dz.\]
As in the discrete case, we say that $\Phi$ is a \textbf{Dirichlet function} (or equivalently that $\Phi$ has \textbf{finite energy}) if it is locally $L^2$, weakly differentiable, and satisfies $\sE(\Phi)<\infty$. We let $\bD(D)$, and  $\HD(D)$ be the spaces of Dirichlet functions (modulo almost everywhere equivalence) and harmonic Dirichlet functions respectively.
The spaces $\bD(D)$ and $\HD(D)$ are Hilbert spaces with respect to the inner product
\begin{equation}
\label{eq:innerproductdefcont}
\langle \Phi,\Psi \rangle 
= \int_{O} \Phi(z) \Psi(z) \dif z + \int_D \nabla \Phi(z) \cdot \nabla \Psi(z) \dif z,
\end{equation}
where $O$ is a fixed precompact open subset of $D$. (The Poincar\'e inequality implies that different choices of $O$ yield equivalent norms. 
In particular, convergence in this norm implies local $L^2$ convergence.)
The spaces 
$\bD(D)$ and $\HD(D)$ contain the spaces of \emph{bounded Dirichlet functions} $\BD(D)$ and of \emph{bounded harmonic Dirichlet functions} $\BHD(D)$ as dense subspaces respectively \cite[Proposition 16]{MR0049396}.

Let $\bD_0(D)$ be the closure in $\bD(D)$ of the space of compactly supported Dirichlet functions. 
As in the discrete case, if $D$ is a transient for Brownian motion, then every $\Phi \in \bD(D)$ has a unique Royden decomposition $\Phi=\Phi_{\bD_0}+\Phi_{\HD}$ where $\Phi_{\bD_0} \in \bD_0(D)$ 
 and $\Phi_{\HD}\in \HD(D)$ \cite{MR0049396}.  
Let $\langle B_t \rangle_{t=0}^{T_{\partial D}}$  be a Brownian motion stopped at the first time it hits $\partial D$, denoted $T_{\partial D}$.
 Anconca, Lyons, and Peres \cite{AnLyPe99} proved that if $\Phi \in \bD(D)$, then the limit $\lim_{t\uparrow T_{\partial D}} \Phi(B_t)$ exists almost surely\footnote{Strictly speaking, since $\Phi$ is only defined up to almost everywhere equivalence, we choose a \emph{quasi-continuous version} of $\Phi$ before applying it to the Brownian motion $B_t$. This ensures that $\Phi(B_t)$ is well-defined and continuous in $t$ almost surely. See \cite{AnLyPe99} for details.}, that 
\begin{equation}
\label{eq:ALPcontinuouslimit}
\lim_{t\uparrow T_{\partial D}}\Phi(B_t) = \lim_{t\uparrow T_{\partial D}}\Phi_{\HD}(B_t)
\end{equation}
almost surely, and that
\begin{equation}
\label{eq:ALPcontinuousidentification}
\Phi_{\HD}(z) = \E_z\left[\lim_{t\uparrow T_{\partial D} } \Phi(B_t) \right]\end{equation}
for every $z\in D$, 
where $\E_z$ denotes the expectation with respect to the Brownian motion $\langle B_t \rangle_{t=0}^{T_{\partial D}}$ started at $z$. The almost sure existence of the limit $\lim_{t\uparrow T_{\partial D}} \Phi(B_t)$ also follows from the earlier work of Doob \cite{MR0109961,MR0173783}.



\subsection{Planar maps and double circle packing}
\label{subsec:mapsdcp}

Let us briefly recall the definitions of planar maps; see e.g.\ \cite{LZ,miermont2014aspects,unimodular2} for detailed definitions.  
Recall that a (locally finite) \textbf{map} $M$ is a connected, locally finite graph $G$ together with an equivalence class of proper embeddings of $G$ into orientable surfaces, where two such embeddings are equivalent if there is an orientation preserving homeomorphism between the two surfaces sending one embedding to the other. Equivalently, maps can be defined combinatorially as graphs equipped  with cyclic orderings of the oriented edges emanating from each vertex, see \cite{LZ} or \cite[Section 2.1]{unimodular2}. We call a graph endowed with \emph{both} a map structure and a network structure (i.e., specified conductances) a \textbf{weighted map}. A map is \textbf{planar} if the surface is homeomorphic to an open subset of the sphere, and is \textbf{simply connected} if the surface is simply connected, that is, homeomorphic to either the sphere or the plane. 

Given a specified embedding of a map $M$, the \textbf{faces} of $M$ are defined to be the connected components of the complement of the embedding. We write $F$ for the set of faces of $M$, and write $f\perp v$ if the face $f$ is incident to the vertex $v$. Given an oriented edge $e$ of $M$, we write $e^\ell$ for the face to the left of $e$ and $e^r$ for the face to the right of $E$.  Every map $M$ has a \textbf{dual} map $M^\dagger$ that has the faces of $M$ as vertices, the vertices of $M$ as faces, and for each oriented edge $e$ of $M$, $M^\dagger$ has an oriented edge $e^\dagger$ from $e^\ell$ to $e^r$. The definitions of $F$ and $M^\dagger$ are  independent of the choice of embedding of $M$, as different embeddings give rise to face sets that are in canonical bijection with each other and dual maps that are canonically isomorphic to each other. It is also  possible to define $F$ and $M^\dagger$ entirely combinatorially, see \cite{LZ} or \cite[Section 2.1]{unimodular2} for details.

The \textbf{carrier} of a circle packing $P$, $\operatorname{carr}(P)$,  is defined to be union of the discs in $P$ together with the components of $\C\cup\{\infty\} \setminus \bigcup P$ whose boundaries are contained in a union of finitely many discs in $P$.  
Note that every circle packing $P$ in the Riemann sphere whose tangency graph is locally finite also defines a locally finite \textbf{tangency map}, where we embed the tangency graph into the carrier of $P$ by drawing straight lines between the centres of tangent circles. 

 \begin{figure}[t]
\centering
\includegraphics[height=0.32\textwidth]{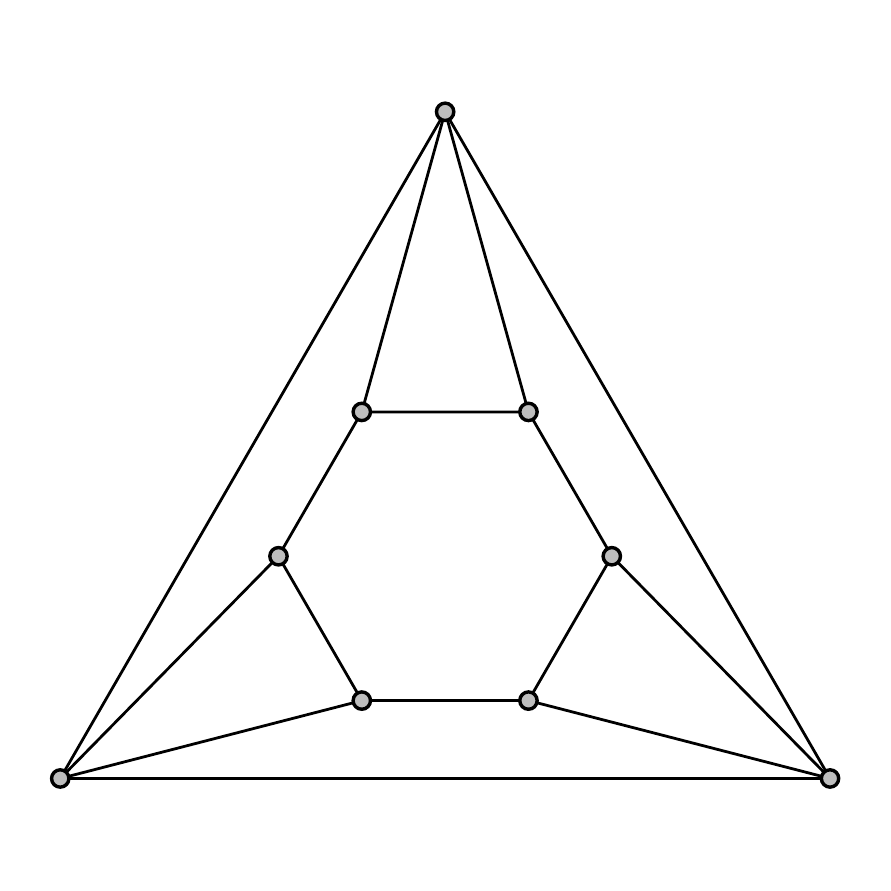} \hspace{1.1cm} \includegraphics[height=0.32\textwidth]{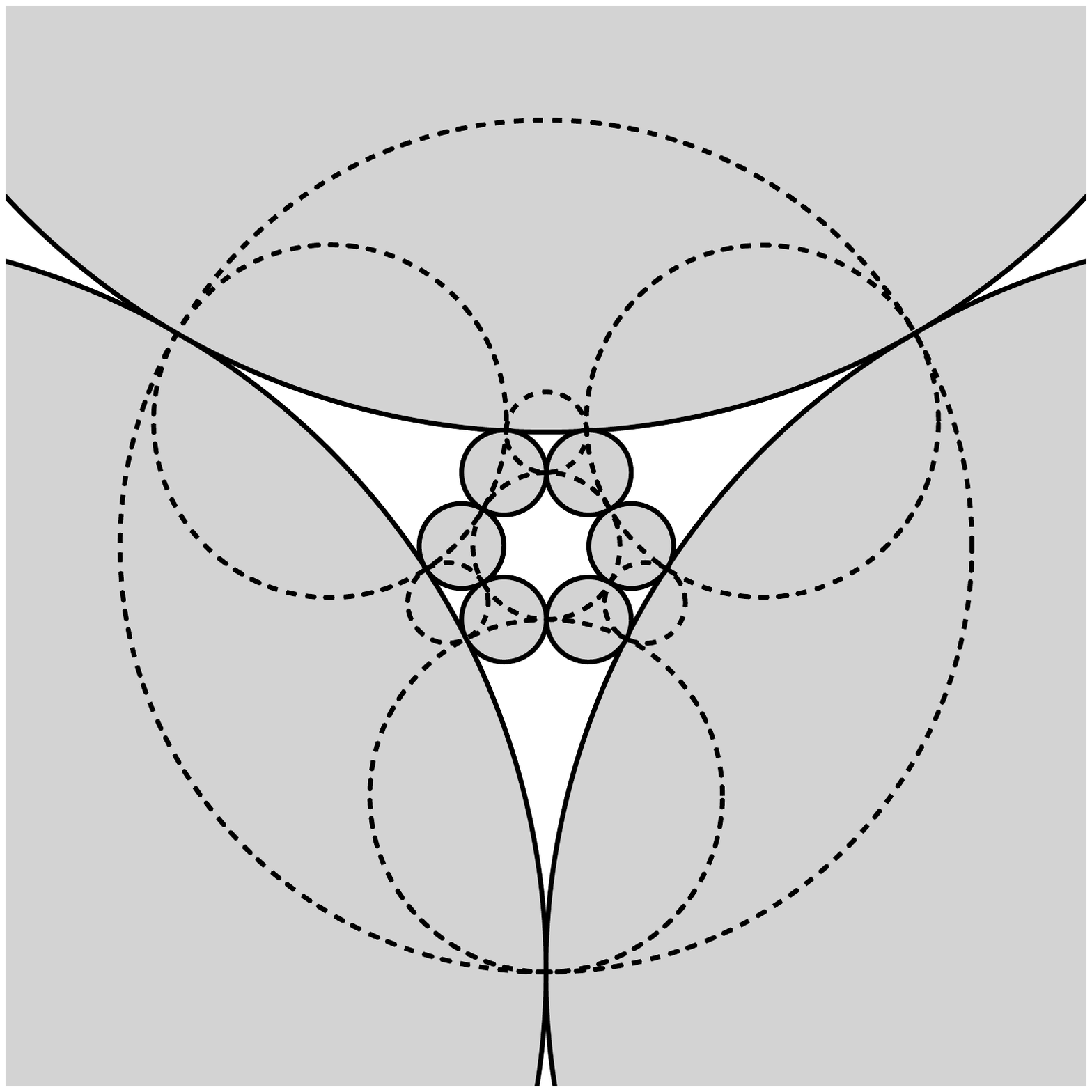}
\caption{
{
A finite polyhedral planar map (left) and its double circle packing (right). Primal circles are filled and have solid boundaries, dual circles have dashed boundaries.}
}
\label{fig.dcp}
\end{figure} 

Let $M$ be a locally finite map with locally finite dual $M^\dagger$. A \textbf{double circle packing} of $M$ is a pair of circle packings $(P,P^\dagger)$ in the Riemann sphere such that the following conditions hold (see \cref{fig.dcp}).
\begin{enumerate}[leftmargin=*]
\item $M$ is the tangency map of $P=\{P(v) : v \in V\}$ and $M^\dagger$ is the tangency map of $P^\dagger=\{P^\dagger(f) : f \in F\}$. 
\item If $v$ is a vertex of $M$ and $f$ is a face of $M$, then the discs $P(v)$ and $P^\dagger(f)$ intersect if and only if $v$ is incident to $f$, and in this case their boundaries intersect orthogonally.
\end{enumerate} 
Observe that if $(P,P^\dagger)$ is a double circle packing of a locally finite map with locally finite dual then $\operatorname{carr}(P)=\operatorname{carr}(P^\dagger)=\bigcup P \cup \bigcup P^\dagger$.
It follows from Thurston's interpretation \cite{Th78,marden1990thurston} of Andreev's theorem \cite{andreev1970convex} that a finite planar map has a double circle packing in the Riemann sphere if and only if it is \textbf{polyhedral}, that is, simple and $3$-connected. The corresponding infinite theory\footnote{He worked in a more general setting, see \cite[Section 2.5]{HutNach15b} for a discussion of how his results imply those claimed here.} was developed by He \cite{he1999rigidity}, who proved that every simply connected, locally finite, polyhedral map $M$ with locally finite dual admits a double circle packing in either the plane or the disc, and that this packing is unique up to M\"obius transformations. (Note that reflections are no longer needed now that we are considering maps instead of graphs.)  See \cite{HS93} for a related uniformization theorem for \emph{countably-connected} triangulations. Without any topological assumptions, we still have by an easy compactness argument that every locally finite polyhedral planar map with locally finite dual admits a double circle packing in \emph{some} domain, although possibly a very wild one. 

\subsection{The isomorphism}

We are now ready to describe our isomorphism theorem in its full generality. We say that a weighted map (or more generally a network)  has \textbf{bounded local geometry} if it has bounded degree and the conductances of its edges are bounded between two positive constants. We say that a map has \textbf{bounded codegree} if its dual has bounded degree. 


\begin{theorem}[The isomorphism]
\label{thm:isomorphismgeneral}
Let $M$ be a transient weighted polyhedral planar map with bounded codegrees and bounded local geometry, let $(P,P^\dagger)$ be a double circle packing of $M$ in a domain $D \subset \C \cup \{\infty\}$, and let $z:V\to D$ be the function sending each vertex $v$ to the centre of the corresponding disc $P(v)$.
  Then the following hold:
\begin{enumerate}
	\item For every harmonic Dirichlet function $h \in \HD(M)$, there exists a unique harmonic Dirichlet function $H \in \HD(D)$  such that $h-H\circ z \in \bD_0(M)$. 
	We denote this function $H$ by $\Cont[h]$.
	\item For every harmonic Dirichlet function $H \in \HD(D)$, there exists a unique harmonic Dirichlet function $h \in \HD(M)$ such that $h-H \circ z \in \bD_0(M)$. 
	We denote this function $h$ by $\Disc[H]$.
\end{enumerate}
	 Moreover, the functions $\mathsf{Cont}:\HD(M)\to \HD(D)$ and $\mathsf{Disc}:\HD(D)\to\HD(M)$ are bounded linear operators, 
	and these operators are inverses of each other.
\end{theorem}

Note that even in the simply connected case there are many choices of domain $D$ and double circle packing $(P,P^\dagger)$ for any given map $M$, and the theorem should be understood as giving us an isomorphism for each such choice of $D$ and $(P,P^\dagger)$.

There are several ways to characterize the space $\bD_0(G)$, leading to several alternative characterisations of the functions $\Cont[h]$ and $\Cont[H]$. In particular, the following hold under the assumptions of \cref{thm:isomorphismgeneral}:
\begin{itemize}
\item For each $h \in \HD(M)$, $H=\Cont[h]$ is the unique harmonic Dirichlet function on $D$ such that 
\begin{equation}
\label{eq:othercharswalk}
\lim_{n\to\infty} \big|h(X_n) - H \circ z(X_n)\big|=0 
\end{equation}
almost surely when $\langle X_n \rangle_{n \geq0}$ is a random walk on $G$. Similarly, for each $H\in \HD(D)$, $h=\Disc[H]$ is the unique harmonic Dirichlet function on $M$ such that \eqref{eq:othercharswalk} holds almost surely. Given \cref{thm:isomorphismgeneral}, both statements are implied by \eqref{eq:ALPlimitdisc}. 
\item
For each $h\in \HD(M)$, 
$H=\Cont[h]$ is the unique harmonic Dirichlet function on $D$ such that $h$ and $H\circ z$ are \textbf{quasi-asymptotically equal}, meaning that
\begin{equation}
\label{eq:othercharscap} 
\Cap\big(\big\{v\in V : |h(v)-H\circ z(v)| \geq \eps\big\}\big)<\infty
\end{equation}
 for every $\eps>0$. 
See \cref{subsec:capacity} for the definition of capacity. 
 Similarly, for each $H\in \HD(D)$, $h=\Disc[H]$ is the unique harmonic Dirichlet function on $M$ such that $h$ is quasi-asymptotically equal to $H \circ z$. Given Theorem \ref{thm:isomorphismgeneral}, both statements are implied by Proposition \ref{lem:D0char}. 
\end{itemize}

We can get the stronger characterisation of $\Cont$ and $\Disc$ in terms of asymptotic equality if we make additional assumptions on the domain. 
We say that a domain $D$ is \textbf{uniformly transient} if
\[
\inf_{z\in D} \Cap\Big(B\big(z,\eps d(z,\partial D)\big)\Big)>0
\]
for every $\eps>0$. 
For example, the unit disc is uniformly transient, as is any finitely connected domain none of whose complementary components are points. 
\begin{itemize}
	\item \emph{If $D$ is uniformly transient}, then for each \emph{bounded} $h\in \BHD(M)$, $H=\Cont[h]$ is the unique harmonic Dirichlet function on $D$ such that $h$ and $H\circ z$ are asymptotically equal. 
 Similarly, for each \emph{bounded} $H\in \BHD(D)$, $h=\Disc[H]$ is the unique harmonic Dirichlet function on $M$ such that $h$ is asymptotically equal to $H \circ z$. As we will see, given \cref{thm:isomorphismgeneral}, both statements are implied by Proposition \ref{prop:uniformlytransient}, and yield \cref{thm:isomorphismdisc} as a special case.
\end{itemize}
Note that the weighted map $M$ is \emph{not} required to be uniformly transient.

\subsection{Related work and an alternative proof}

A related result concerning linear isomorphisms between harmonic Dirichlet spaces induced by rough isometries between bounded degree graphs was shown by Soardi \cite{Soardi93}, who proved that if $G_1$ and $G_2$ are bounded degree, rough isometric graphs, then $G_1$ admits non-constant harmonic Dirichlet functions if and only if $G_2$ does. See e.g.\ \cite{Soardibook,LP:book} for definitions of and background on rough isometries. Soardi's result was subsequently generalized by Holopainen and Soardi \cite{holopainen1997p} to rough isometries between bounded degree graphs and a certain class of Riemannian manifolds. This result was then strengthened by Lee \cite{lee1999rough}, who showed that the dimension of the space of harmonic Dirichlet functions is preserved under rough isometry.

 By a small improvement on the methods in the works mentioned (or, alternatively, using the methods of this paper), it is not difficult to show the stronger result that for each rough isometry $\rho : G_1 \to G_2$, we have that $h \mapsto (h \circ \rho)_{\HD}$ is a bounded linear isomorphism $\HD(G_2)\to\HD(G_1)$. 
 Similar statements hold for rough isometries between graphs and manifolds  and between two manifolds (under appropriate assumptions on the geometry in both cases). Indeed, in the discrete case the fact that $h \mapsto (h \circ \rho)_{\HD}$ is a bounded linear isomorphism can easily be read off from the proof of Soardi's result presented in \cite{LP:book}. 

 Another setting in which one very easily obtains an isomorphism between harmonic Dirichlet spaces is given by quasi-conformal mapping between domains (or other Riemannian manifolds). Recall that a homeomorphism $q:D\to D'$ is said to be \textbf{quasi-conformal} if it is orientation preserving, weakly differentiable, and there exists a constant $C$ such that
 \[\|D q(z)\|^2 \leq C \,|\!\det \left[D q(z)\right]\!| \]
 for a.e.\ $z\in D$. It is trivial to verify by change of variables that $\cE(\phi \circ q) \leq C \cE(\phi)$ for every $\phi \in \bD(D)$ and $\cE(\psi \circ q^{-1}) \leq C \cE(\psi)$ for every $\psi \in \bD(D')$, so that composition with $q$ defines a bounded linear isomorphism from $\bD(D')$ to $\bD(D)$. Moreover, it is immediate that $\psi \circ q \in \bD_0(D)$ if and only if $\psi \in \bD_0(D')$, and it follows that $H \mapsto (H\circ q)_{\HD}$ is a bounded linear isomorphism from $\HD(D')$ to $\HD(D)$.

Using these ideas, one could obtain an alternative, less direct proof of \cref{thm:isomorphismgeneral}, sketched as follows: First, let $S$ be the `piecewise flat' surface obtained by gluing regular polygons according to the combinatorics of the map $M$, which is Riemannian apart from having conical singularities at its vertices.  The assumption that $M$ has bounded degrees and codegrees readily implies that the function $i$ sending each vertex of $M$ to the corresponding point of $S$ is a rough isometry. One can then show that $H \mapsto (h \circ i)_{\HD}$ is a bounded linear isomorphism $\HD(S)\to\HD(M)$, similar to the above discussion. Next, the Ring Lemma easily allows us to construct, face-by-face, a quasi-conformal map $q:S\to D$ such that $q\circ i = z$. One can then arrive at \cref{thm:isomorphismgeneral} by composing the isomorphism $\HD(S)\to\HD(M)$, $H \mapsto (H \circ i)_{\HD}$  and the isomorphism $\HD(D)\to\HD(S)$, $H \mapsto (H \circ q)_{\HD}$.

\section{Proof}

\subsection{Capacity characterisation of $\bD_0$}
\label{subsec:capacity}

Recall that the \textbf{capacity} of a finite set of vertices $A$ in a network $G$ is defined to be 
\[\Cap(A)=\sum_{v\in A} c(v) \bP_v(\tau^+_A =\infty),\]
where $\bP_v(\tau^+_A=\infty)$ is the probability that a random walk on $G$ started at $A$ never returns to $A$ after time zero and $c(v)=\sum_{e\in E^\rightarrow : e^- =v} c(e)$ is the total conductance of all oriented edges emanating from the vertex $v$. The capacity of an infinite set $A$ is defined to be $\Cap(A)=\sup\{\Cap(A') : A' \subseteq A \text{ finite}\}.$ 
Another way to compute capacities is via 
\textbf{Dirichlet's principle}, which gives the following variational formula for the capacity of a (finite or infinite) set $A$ in a network $G$ (see e.g.\ \cite[Chapter 2]{LP:book}):
\[
\Cap(A)=\inf\left\{\cE(\phi): \phi \in \bD_0(G),\, \phi|_A \geq 1 \right\},
\]
where we set $\inf \emptyset = \infty$. (For example, if $G=(V,E)$ is transient then $\Cap(V)=\infty$ and the set $\{\phi \in \bD_0(G),\, \phi|_V \geq 1\}$ is empty.)
A similar formula can also be taken as the definition of the capacity of a set $A$ in a domain $D$ (see e.g.\ \cite{AnLyPe99}):
\[
\Cap(A):= \inf\big\{\cE(\Phi): \Phi \in \bD_0(D),\,  \Phi \geq 1 \text{ a.e.\ on an open  neighbourhood of $A$}\big\}.
\]
A network is transient if and only if some (and hence every) finite set of its vertices has positive capacity, and a 
domain is transient if and only if some (and hence every) precompact open subset of it has positive capacity.

The following characterisation of $\bD_0$ is presumably well-known to experts.

\begin{prop}
\label{lem:D0char}
\hspace{1cm}
\begin{enumerate}
	\item
Let $G$ be a network and let $\phi \in \bD(G)$. Then $\phi \in \bD_0(G)$ if and only if it is quasi-asymptotically equal to the zero function, that is, if and only if
\[
  \Cap\left(\{v\in V : |\phi(v)| \geq \eps \}\right)<\infty
\]
for every $\eps>0$.
\item
Let $D$ be a domain and let $\Phi \in \bD(D)$. Then $\Phi \in \bD_0(D)$ if and only if it is quasi-asymptotically equal to the zero function, that is, if and only if
\[
  \Cap\big(\{z\in D : |\Phi(z)| \geq \eps \text{ a.e.\ on an open neighbourhood of $z$} \}\big)<\infty.
\]
for every $\eps>0$.
\end{enumerate}
\end{prop}

\begin{proof}
We prove item $1$; item $2$ is similar. If $G$ is recurrent, then $\bD_0(G)=\bD(G)$ \cite[Theorem 3.63]{Soardibook} and every set has capacity zero, so that the claim holds vacuously. Thus, it suffices to consider the case that 
 $G$ is transient.  
Let $\phi \in \bD(G)$. If $\phi \in \bD_0(G)$ then
for each $\eps>0$, the function $\psi=\eps^{-1}|\phi|$ satisfies $\psi \geq 1$ on the set $\{v\in V: |\phi(v)| \geq \eps\}$. It is easily verified that $\psi \in \bD_0(G)$ and that $\cE(\psi) \leq \eps^{-2} \cE(\phi)$, and so Dirichlet's principle implies that
\begin{equation}
\Cap\big(\{v\in V: |\phi(v)| \geq \eps\}\big) \leq \cE(\psi) \leq \eps^{-2} \cE(\phi) <\infty
\end{equation}
as claimed. Conversely, suppose that $\Cap(\{v\in V : |\phi(v)| \geq \eps\})<\infty$ for every $\eps>0$. Then for every $\eps>0$ there exists $\psi_\eps \in \bD_0(G)$ such that $\psi_\eps \geq 1$ on the set $\{v\in V : |\phi(v)|\geq \eps\}$. Let $\langle X_n \rangle_{n \geq0}$ be a random walk on $M$. We deduce from the uniqueness of the Royden decomposition \eqref{eq:Royden} and from \eqref{eq:ALPlimitdisc} and \eqref{eq:ALPlimitidentification} that $\lim_{n\to\infty} \psi_\eps(X_n) =0$ almost surely, and hence that $\limsup_{n\to\infty} |\phi(X_n)| \leq \eps$ almost surely. Since $\eps>0$ was arbitrary it follows that $\lim_{n\to\infty} \phi(X_n) =0$ almost surely, and we deduce from \eqref{eq:ALPlimitidentification} that $\phi \in \bD_0(G)$ as claimed. \qedhere
\end{proof}

\subsection{Proof of the main theorems}

We begin by recalling the Ring Lemma of Rodin and Sullivan \cite{RS87}, which was originally proven for circle packings of triangulations and was generalized to double circle packings of polyhedral maps in \cite{HutNach15b}. See \cite{Hansen,Ahar97} for quantitative versions in the case of triangulations. Given a double circle packing $(P,P^\dagger)$ in a domain $D \subseteq \C$ of a map $M$ we write $r(v)$ for the radius of $P(v)$ and $r(f)$ for the radius of $P^\dagger(f)$ for each $v\in V$ and $f\in F$.

\begin{theorem}[The Ring Lemma]
There exists a family of positive constants $\langle k_{n,m} : n\geq 3, m\geq 3\rangle$ such that if $(P,P^\dagger)$ is a double circle packing of a polyhedral planar map $M$ in a domain $D \subseteq \C$, then
\[r(v)/r(f) \leq k_{\deg(v),\max_{g \perp v}\deg(g)}\]
for every vertex $v\in V$ and every $f\in F$ incident to $v$.
\end{theorem}

For the rest of this section $M$ will be a transient weighted polyhedral map with bounded codegrees and bounded local geometry, $(P,P^\dagger)$ will be a double circle packing of $M$ in a domain $D \subseteq \C \cup \{\infty\}$, and $z$ will be the associated embedding of $M$. By applying a M\"obius transformation if necessary, we can and will assume that $D \subseteq \C$ (in which case $D\subsetneq \C$ by the He-Schramm theorem since $M$ is transient). 
We write $\bM=\bM(M)$ for the data
\[\bM(M) = \big(\max_{v\in V} \deg(v),\; \max_{f\in F} \deg(f),\; \sup_{e\in E} c(e),\; \sup_{e\in E} c^{-1}(e)\big).\] 
We say that two quantities are \textbf{comparable} if they differ up to positive multiplicative constants depending only on $\bM$, and write $\asymp$, $\preceq$, and $\succeq$  for equalities and inequalities that hold up to positive multiplicative constants depending only on the data $\bM$. 
We also use standard big-O notation, where again the implicit positive multiplicative constants depend only on $\bM$.

A consequence of the Ring Lemma is that the embedding of $M$ given by drawing straight lines between the centres of circles in its double circle packing is \emph{good}\footnote{We remark that all our results hold more generally for good straight-line embeddings of $M$, not just those produced using double circle packing. However, we are not aware of any general method of producing good embeddings that does not rely on double circle packing.} in the sense of \cite{ABGN14}, meaning that adjacent edges have comparable lengths and that the faces in the embedding have internal angles uniformly bounded away from zero and $\pi$. 
We will require the following useful geometric property of good embeddings of planar graphs, stated here for double circle packings.
For each $v\in V$ and $\delta>0$, we write $P_{\delta}(v)$ for the disc that has the same centre as $P(v)$ but has radius $\delta r(v)$. Given a set of vertices $A \subseteq V$, we write $P_{\delta}(A)$ for the union $P_{\delta}(A)=\bigcup_{v\in A} P_{\delta}(v)$.

	\begin{lemma}[The Sausage Lemma \cite{ABGN14}]
There exists a positive constant $\delta_1=\delta_1(\bM)$ such that for each two oriented edges $e_1,e_2\in E^\rightarrow$ of $M$ that do \emph{not} share an endpoint, the convex hull of $P_{\delta_1}(e^-_1)\cup P_{\delta_1}(e_1^+)$ and the convex hull of $P_{\delta_1}(e_2^-)\cup P_{\delta_1}(e^+_2)$ are disjoint.
\end{lemma}

We now define the two operators that will be the key players in the proof of \cref{thm:isomorphismgeneral}.

\begin{definition}[The operator $\mathsf{R}$]
Fix $\delta_0=\delta_0(\bM) \leq 1/2$ sufficiently small that $\delta_0$ is less than or equal to the sausage lemma constant $\delta_1$ and that $\frac{1}{4}|z(u)-z(v)| \geq \delta_0 r(v)$ for every adjacent pair $u,v\in V$.
 For each locally integrable $\Phi : D\to \R$, we define $\mathsf{R}[\Phi]:V\to \R $ by setting $\mathsf{R}[\Phi](v)$ to be the average value of $\Phi$ on the disc $P_{\delta_0}(v)$ for each $v\in V$, that is,
\[ \mathsf{R}[\Phi](v) = \frac{1}{\pi \delta_0^2 r(v)^2}\int_{P_{\delta_0}(v)} \Phi(z) \dif z.\]
\end{definition}
If $H\in \HD(D)$, then it follows from harmonicity that
$\mathsf{R}[H](v) = H \circ z(v)$ for every $v\in V$.

\begin{definition}[The operator $\mathsf{A}$] 
Consider the triangulation $T$ embedded with straight lines in $D$ that is obtained by drawing a straight line between $z(v)$ and $z(u)$ whenever $u$ and $v$ are adjacent vertices of $M$, and a straight line between $z(v)$ and $z(f)$ (the centre of $P^\dagger(f)$) whenever $v$ is a vertex of $M$ and $f\perp v$ is a face of $M$ incident to $v$. 
    For each function $\phi:V\to \R$, we define the \textbf{piecewise-affine extension} $\mathsf{A}[\phi]$ of $\phi$ to $D$
       to be
     the unique function on $D$ that takes the values 
     \[\mathsf{A}[\phi](z(v)) = \phi(v)  \text{ for every } v\in V\quad \text{ and } \quad \mathsf{A}[\phi](z(f))=\phi(f):=\frac{1}{\deg(f)}\sum_{v \perp f} \phi(v)  \text{ for every } f \in F  \]
     on $z(V)=\{z(v): v\in V\}$ and $z(F)=\{z(f): f \in F\}$,       and is affine on each edge and each face of the triangulation $T$.  
\end{definition}

 We fix a root vertex $o$ of $M$ with which to define the inner product on $\bD(M)$ in \eqref{eq:innerproductdefdisc}, and take the interior of $P_{\delta_0}(o)$ to be the precompact open set $O$ used to define the inner product on $\bD(D)$ in \eqref{eq:innerproductdefcont}.

\begin{lemma}
\label{lem:AandRenergy}
$\mathsf{R}:\bD(D)\to\bD(M)$
and $\mathsf{A}:\bD(M)\to\bD(D)$ 
are bounded linear operators with norms bounded by constants depending only on $\bM(M)$, and also satisfy
\[\cE(\mathsf{R}[\Phi]) \preceq \cE(\Phi) \quad \text{ and } \quad \cE(\mathsf{A}[\phi]) \preceq \cE(\phi)\]
for every $\Phi \in \bD(D)$ and $\phi \in \bD(M)$. 
In particular, $\mathsf{R}[\Phi]\in \bD(M)$ for every $\Phi \in \bD(D)$ and $\mathsf{A}[\phi]\in \bD(D)$ for every $\phi\in \bD(M)$.
\end{lemma}

The main estimates needed for this lemma are implicit in
 \cite{GGNS15}, and our proof is closely modeled on the arguments in that paper.

\begin{proof}[Proof of \cref{lem:AandRenergy}]
We begin with $\mathsf{A}$. We wish to show that $\cE(\mathsf{A}[\phi])\preceq \cE(\phi)$. Let $\phi\in \bD(M)$, let $e\in E^\rightarrow$ be an oriented edge of $M$,  and let $T_{e}$ be the triangle with corners at $z(e^-),z(e^+),$ and $z(e^\ell)$. 
For each $e\in E^\rightarrow$,
let $\psi_e$ be the linear map sending $T_{e}$ to the convex hull of $\{0,1,i\}$ that sends $e^\ell$ to $0$, $e^-$ to $1$, and $e^+$ to $i$. It follows from the Ring Lemma that $\|\mathrm{D} \psi_e(z)\|\asymp r(e^-)^{-1}$ for all $z\in T_e$, where $\mathrm{D}\psi_e$ denotes the total derivative of $\psi_e$. On the other hand, $\mathsf{A}[\phi]\circ \psi_e^{-1}$ is equal to the affine function $x+iy \mapsto (1-x-y)\phi(e^\ell) + x\phi(e^-) + y \phi(e^+)$, and we deduce that
\begin{align*}\|\nabla \mathsf{A}[\phi](z)\| &\leq \|\mathrm{D} \psi_e(z)\|\left\|\nabla \!\left(\mathsf{A}[\phi] \circ \psi_e^{-1}\right)(\psi_e(z))\right\|\\
&\asymp r(e^-)^{-1} \max \bigl\{|\phi(e^-)-\phi(e^+)|,\, |\phi(e^-)-\phi(e^\ell)|,\, |\phi(e^+)-\phi(e^\ell)|\bigr\}.
\end{align*}
Integrating over $z \in T_{e}$ and summing over $e\in E^\rightarrow$,  we obtain that
\begin{align}
\cE(\mathsf{A}[\phi]) &= \sum_{e\in E^\rightarrow} \int_{T_{e}} \|\nabla \mathsf{A}[\phi](z)\|^2 \dif z \preceq \sum_{e\in E^\rightarrow} 
\max \big\{|\phi(e^-)-\phi(e^+)|,\, |\phi(e^-)-\phi(e^\ell)|,\, |\phi(e^+)-\phi(e^\ell)|\big\}^2
\nonumber
\\
&\preceq 
\sum_{e\in E^\rightarrow} |\phi(e^-)-\phi(e^+)|^2 \quad + \sum_{v \in V, f\in F, f \perp v} |\phi(v)-\phi(f)|^2,
\label{eq:EAtwoterms}
\end{align}
where in the first inequality we have used the fact that, by the Ring Lemma, the area of $T_{e}$ is comparable to $r(e^-)^2$ for every $e\in E^\rightarrow$. Now, for each face $f$ of $M$, we have that
\[\max_{u,v \perp f} |\phi(u)-\phi(v)| \leq \sum_{e:e^\ell=f}|\phi(e^+)-\phi(e^-)|,\]
and hence by Cauchy-Schwarz we have that
\begin{align}
\sum_{v \in V, f\in F, f \perp v} |\phi(v)-\phi(f)|^2
&\leq 
\sum_{v \in V, f\in F, f \perp v} \max_{u \perp f} |\phi(u)-\phi(v)|^2
\leq 
\sum_{v \in V, f\in F, f \perp v} \left[\sum_{e:e^\ell=f}|\phi(e^+)-\phi(e^-)|\right]^2
\nonumber
\\
 &\leq
  \sum_{v \in V, f \in F, f \perp v} \deg(f) \sum_{e: e^\ell=f} |\phi(e^+)-\phi(e^-)|^2.
  \label{eq:EAsecondterm}
\end{align}
Since each oriented edge is counted at most a constant number of times in this sum we obtain from \eqref{eq:EAtwoterms} and \eqref{eq:EAsecondterm} that
\begin{equation}
\label{eq:Aenergyfinal}
\cE(\mathsf{A}[\phi]) \preceq \sum_{e\in E^\rightarrow} |\phi(e^+)-\phi(e^-)|^2 \preceq \cE(\phi)
\end{equation}
as required. To control the other term in $\langle \mathsf{A}[\phi],\mathsf{A}[\phi]\rangle$, observe that
\begin{align*}
 \int_{P_{\delta_0}(o)} \mathsf{A}[\phi](z)^2 \dif z &\preceq \max\big\{|\phi(u)|^2 : u \text{ shares a face with $o$}\big\}\\ &\preceq \phi(o)^2 + \max\big\{|\phi(u)-\phi(o)|^2 : u \text{ shares a face with $o$}\big\}, \end{align*}
 where we say that two vertices $u$ and $v$ \textbf{share a face} if there exists $f\in F$ such that $u\perp f$ and $v\perp f$. 
A simple Cauchy-Schwarz argument similar to the above then shows that
\begin{equation}
\label{eq:Anormother}
\int_{P_{\delta_0}(o)} \mathsf{A}[\phi](z)^2 \dif z \preceq \phi(o)^2 + \cE(\phi),
\end{equation}
and combining \eqref{eq:Aenergyfinal} and \eqref{eq:Anormother} yields that $\langle \mathsf{A}[\phi],\mathsf{A}[\phi]\rangle \preceq \langle \phi,\phi \rangle$ as required.

\medskip

We now show that $\mathsf{R}$ is bounded. 
 We wish to show that $\langle \mathsf{R}[\Phi],\mathsf{R}[\Phi]\rangle \preceq \langle \Phi,\Phi\rangle$ and moreover that $\cE(\mathsf{R}[\Phi]) \preceq \cE(\Phi)$ for every $\Phi\in \bD(D)$. Let us first suppose that $\Phi$ is continuously differentiable. It is well known, and can be seen by a simple mollification argument, that such $\Phi$ are dense in $\bD(D)$ (as indeed are the smooth Dirichlet functions). For each $v\in V$, let $X_v$ be a random point chosen uniformly from the disc $P_{\delta_0}(v)$, independently from each other,
 so that
$\mathsf{R}[\Phi](v)=\E \Phi(X_v)$.
  For each $u,v\in V$, let $\Gamma_{u,v}$ be the random line segment connecting $X_u$ to $X_v$. By Jensen's inequality and the assumption that $\Phi$ is continuously differentiable we have that
  \[
  \left(\mathsf{R}[\Phi](u)-\mathsf{R}[\Phi](v)\right)^2 
  =
  \E\left[
  \Phi(X_u)-\Phi(X_v)
   \right]^2
   \leq
   \E\left[
  \left(\Phi(X_u)-\Phi(X_v)\right)^2
   \right] =\E \Big[ \big(\int_{\Gamma_{u,v}}\!\|\nabla \Phi(z)\|\dif z\big)^2 \Big].
  \]
For each adjacent $u,v \in V$, conditional on $\Gamma_{u,v}$, let $Z_{u,v}$ be a random point chosen uniformly on the line segment $\Gamma_{u,v}$. The Cauchy-Schwarz inequality implies that
  \[
\Bigl(\int_{\Gamma_{u,v}}\!\|\nabla \Phi(z)\|\dif z\Bigr)^2 \leq 
|\Gamma_{u,v}| \int_{\Gamma_{u,v}} \|\nabla \Phi(z)\|^2 \dif z
\leq
|\Gamma_{u,v}|^2 \,\E\left[ \|\nabla \Phi(Z_{u,v})\|^2 \mid \Gamma_{u,v}\right].
  \]
    Next, the Ring Lemma implies that $|\Gamma_{u,v}|\preceq r(v)$, and we deduce that
    \begin{equation}
    \label{eq:RZbound}
    \left(\mathsf{R}[\Phi](u)-\mathsf{R}[\Phi](v)\right)^2  \leq 
\E\left[\Bigl(\int_{\Gamma_{u,v}}\!\|\nabla \Phi(z)\|\dif z\Bigr)^2\right] \preceq r(v)^2 \E\left[ \|\nabla \Phi(Z_{u,v})\|^2 \right].
    \end{equation}

    Let $\mu_{u,v}$ be the law of $Z_{u,v}$ and let $A_{u,v}$ be its support, i.e., the convex hull of $P_{\delta_0}(u)\cup P_{\delta_0}(v)$. We claim that the Radon-Nikodym derivative of $\mu_{u,v}$ with respect to the Lebesgue measure on $A_{u,v}$ is
    $O(r(v)^{-2})$. This is equivalent to the claim that
    \begin{equation}
    \label{eq:RadonNikodym}
    \P\left(Z_{u,v} \in B(z,\delta r(v))\right)\preceq  \delta^2 
    \end{equation}
    for every $z \in A_{u,v}$ and $\delta>0$. Suppose without loss of generality that $|z-z(v)|\leq |z-z(u)|$, and condition on the value of $X_u$, so that $|X_u-z|\geq |z(u)-z(v)|/4\succeq r(v)$ by definition of $\delta_0$.  In order for $Z_{u,v}$ to be in the ball $B(z,\delta r(v))$, we must have that $X_v$ is in the cone $K$ that has its vertex at $X_u$ and that is tangent to $B(z,\delta r(v))$, see \cref{fig:cone}. Since $|X_u-z|\geq |z(u)-z(v)|/4$, it follows by elementary trigonometry that the internal angle at the vertex of $K$ is $O(\delta)$, and consequently that the intersection of $K$ with $P_{\delta_0}(v)$ (or indeed with all of $A_{u,v}$), being contained inside a triangle with height $O(r(v))$ and width $O(\delta r(v))$, has area at most  $O(\delta r(v)^{2})$. Thus, the probability that $X_v$ lies in this region is at most $O(\delta)$. Conditioned on the event that $X_v$ lies in $K$, the intersection of $\Gamma_{u,v}$ with $B(z,\delta)$ has length at most $2\delta r(v)$, and so the conditional probability that $Z_{u,v}$ lies in this segment is $O(\delta)$. The estimate \eqref{eq:RadonNikodym} follows.

    \begin{figure}[t]
\centering
\includegraphics[width=0.65\textwidth]{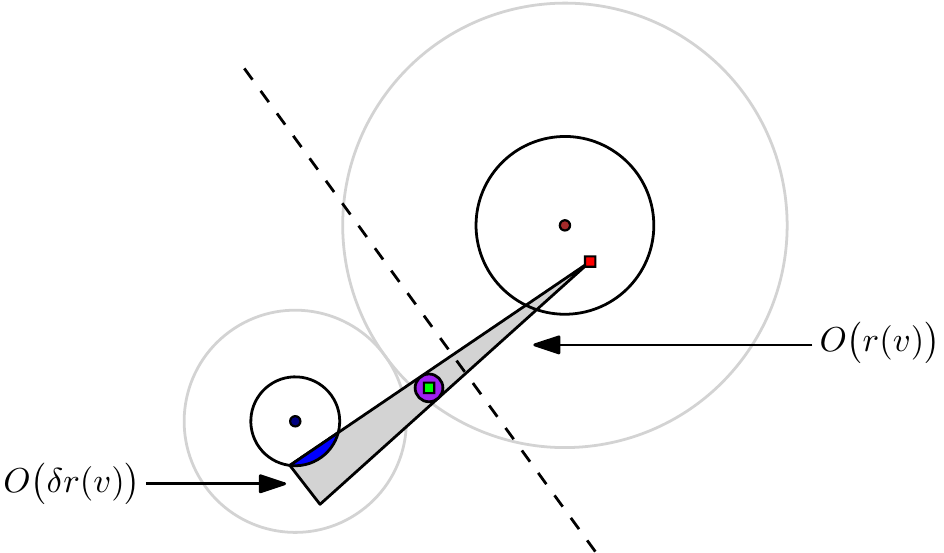}

\caption{\small{Illustration of the proof of the boundedness of $\mathsf{R}$. Suppose that $z$ (green square) is closer to $z(v)$ (navy disc) than to $z(u)$ (brown disc). Then conditional on the location of $X_u$ (red square), in order for $Z_{u,v}$ to be located in $B(z,\delta r(v))$ (purple disc),  $X_v$ must be located in the intersection (blue segment) of $P_{\delta_0}(v)$ with the cone whose vertex is at $X_u$ and that is tangent to $B(z,\delta r(v))$. The dashed line is the perpendicular bisector of the line from $z(u)$ to $z(v)$. This intersection is contained within a triangle (grey) whose sides have lengths of order $O(r(v))$, $O(r(v))$ and $O(\delta r(v))$, and consequently has area $O(\delta r(v)^2)$.}
}
\label{fig:cone}
\end{figure}

    Integrating over the Radon-Nikoydm estimate \eqref{eq:RadonNikodym} we obtain that
    \[
    \E\left[ \|\nabla \Phi(Z_{u,v})\|^2 \right] =
   \int_{A_{u,v}}\frac{d\mu_{u,v}(z)}{\dif z} \| \nabla \Phi(z) \|^2  \dif z
     \preceq r(v)^{-2} \int_{A_{u,v}} \| \nabla \Phi(z) \|^2 \dif z
    \]
    and hence by \eqref{eq:RZbound} that
    \begin{equation}
    \label{eq:Rmainestimate}
    \left(\mathsf{R}[\Phi](u)-\mathsf{R}[\Phi](v)\right)^2 \preceq  \int_{A_{u,v}} \| \nabla \Phi(z) \|^2 \dif z
    \end{equation}
    for every adjacent $u,v\in V$. Since \eqref{eq:Rmainestimate} holds uniformly for all continuously differentiable $\Phi\in \bD(D)$ and the expressions on both sides of the inequality are  continuous functions of $\Phi \in \bD(D)$, we deduce by density that the inequality holds for \emph{all} $\Phi \in \bD(D)$.

    Since $\delta_0$ was taken to be less than the Sausage Lemma constant, we have that each point $z$ is in at most $\max_{v\in V}\deg(v)=O(1)$ different regions of the form $A_{u,v}$, so that applying \eqref{eq:Rmainestimate} yields that
    \begin{equation}
    \label{eq:ERfinalbound}
    \cE(\mathsf{R}[\Phi]) =  \sum_{e\in E^\rightarrow} \left(\mathsf{R}[\Phi](e^-)-\mathsf{R}[\Phi](e^+)\right)^2 \preceq \sum_{e\in E^\rightarrow} \int_{A_{e^-,e^+}} \| \nabla \Phi(z)\|^2 \dif z \preceq \int_D \| \nabla \Phi(z)\|^2 \dif z = \cE(\Phi)
    \end{equation}
    as required.  The other term in $\langle \mathsf{R}[\Phi],\mathsf{R}[\Phi]\rangle$ can be bounded using Jensen's inequality,
    which yields that
\begin{equation}
\label{eq:normRotherbound}|\mathsf{R}[\Phi](o)|^2 \preceq \int_{P_{\delta_0}(o)}\Phi^2(z) \dif z. 
\end{equation}
Combining  \eqref{eq:ERfinalbound} and \eqref{eq:normRotherbound} yields that $\langle \mathsf{R}[\Phi],\mathsf{R}[\Phi]\rangle \preceq \langle \Phi,\Phi\rangle$ as required. 
    \qedhere

\end{proof}

It is an immediate consequence of the closed graph theorem that if a Banach space $V$ is written as the direct sum of two closed subspaces $V=V_1 \oplus V_2$ then the associated projections onto each of the subspaces are bounded. (This can also be argued directly.) Applying this fact in our setting we obtain that the projections $\phi \mapsto \phi_{\HD}$ and $\Phi \mapsto \Phi_{\HD}$ are bounded. Thus, it follows as an immediate corollary to \cref{lem:AandRenergy} that the operators $\Disc:\HD(D)\to\HD(M)$ and $\Cont:\HD(M)\to\HD(D)$ defined by
\begin{align}\mathsf{Disc}[H](v) &= (\mathsf{R}[H])_{\HD}(v) =  (H \circ z)_{\HD}(v) = \bE_v\left[\lim_{n\to\infty} H \circ z(X_n) \right]
\qquad &H\in \HD(D),\; v\in V 
\label{eq:discdef}\\
\mathsf{Cont}[h](z) &= \,(\mathsf{A}[h])_{\HD}(z)\, = \E_{z}\left[\lim_{t\to T_{\partial D}} \mathsf{A}[h](B_t) \right] \qquad & h \in \HD(M),\; z \in D \label{eq:contdef}
\end{align}
 are also well defined and bounded. Here the final equalities of \eqref{eq:discdef} and \eqref{eq:contdef} follow from \eqref{eq:ALPlimitidentification} and \eqref{eq:ALPcontinuousidentification} respectively.

\medskip

A second immediate corollary is the following. 
\begin{corollary}
\label{lem:d0tod0easy}
If $\phi \in \bD_0(M)$ then $\mathsf{A}[\phi] \in \bD_0(D)$. Similarly, if $\Phi\in \bD_0(D)$ then $\mathsf{R}[\Phi] \in \bD_0(M)$.
\end{corollary}

\begin{proof}
We prove the first sentence, the second being similar. 
It is immediate from the definitions that if $\phi \in \bD_0(M)$ is 
finitely supported, then $\mathsf{A}[\phi]$ is compactly supported. We conclude by applying the boundedness of $\mathsf{A}$.
\end{proof}

The following lemma, which is proved below and is also an easy corollary of \cref{lem:AandRenergy}, is also implicit in \cite{GGNS15}; indeed, it can be thought of as a quantitative form of the main result of that paper.

\begin{lemma}
\label{lem:CapComparison}
For every $0<\delta \leq 1/2$, we have that
\[
 \delta^4 \Cap(A) \preceq  \Cap(P_{\delta}(A))  \preceq \, \Cap(A)
\]
for every set of vertices $A$ in $M$.
\end{lemma}

We will require the following simple estimates.

\begin{lemma}[Continuity estimates]
\label{lem:continuity}
\hspace{1cm}
\begin{enumerate}
	\item
Let $\phi:V\to \R$ be a function. Then 
\[ \sup_{z\in P_\delta(v)}\big| \mathsf{A}[\phi](z) - \phi(v)\big| \leq \delta\sup \left\{ |\phi(u)-\phi(v)| : \text{$u$ and $v$ share a face of $M$}\right\} 
\preceq \delta \sqrt{\cE(\phi)}
\]
for every $v\in V$ and $0<\delta<1$.
\item 
Let $H:D\to\R$ be a harmonic function. 
 Then for every $r>0$, $\alpha>1$, and $z_0 \in D$ such that $B(z_0,\alpha r) \subseteq D$ we have that
 \[\sup_{z\in B(z_0,r)} |H(z)-H(z_0)|^2
 \leq \frac{1}{\pi} \log\left[ \frac{\alpha^2}{\alpha^2-1} \right] \int_{B(z_0,\alpha r)} \| \nabla H(z) \|^2 \dif z.
 \]
\end{enumerate}
\end{lemma}

\begin{proof}
The first inequality of item $1$ is immediate from the definition of $\mathsf{A}[\phi]$, while the second follows since
\begin{multline*}
\sup \left\{ |\phi(u)-\phi(v)| : \text{$u$ and $v$ share a face of $M$}\right\}  \leq \sup_{f\in F} \sum_{e \in E^\rightarrow : e^\ell = f} |\phi(e^+)-\phi(e^-)|
\\\preceq \sup_{e\in E^\rightarrow} |\phi(e^+)-\phi(e^-)| \preceq \sqrt{\cE(\phi)}.
\end{multline*}
Item $2$ follows by taking $\Phi: B(z_0,r)\to \C$ to be holomorphic with real part $H$ and applying the inequality of \cite[Theorem 1.2.1]{MR3185375} to the function $\Psi:\D\to \C$ defined by $\Psi(z)=\Phi((z_0+z)/\alpha r)$. (Note that their definition of the energy of $\Psi$ disagrees with ours by a factor of $\pi$.)
\end{proof}

\begin{proof}[Proof of \cref{lem:CapComparison}]
We start with the upper bound.
Let $\phi\in \bD_0(M)$ be such that $\phi|_A \geq 1$, and let $\psi = (\phi \wedge 1)\vee 0$. It is easily verified that $\cE(\psi) \leq \cE(\phi)$ and $\psi |_A = 1$, and it follows from \cref{lem:D0char} that $\psi\in \bD_0(M)$ (this is also easy to verify directly).  \cref{lem:continuity} implies that $\mathsf{A}[\psi](z) \geq 1- \delta$ for every $z\in P_\delta(A)$. Thus, by \cref{lem:d0tod0easy}, we have that $2(1-\delta)^{-1}\mathsf{A}[\psi] \in \bD_0(D)$ and that $2(1-\delta)^{-1}\mathsf{A}[\psi] \geq 1$ on an open neighbourhood of $P_\delta(A)$, so that, by Dirichlet's principle and \cref{lem:AandRenergy},
\[\Cap(P_\delta(A)) \leq \cE(2(1-\delta)^{-1}\mathsf{A}[\psi]) \preceq \cE(\psi) \leq \cE(\phi). \] 
The claimed upper bound follows by taking the infimum over $\phi$.

We now turn to the lower bound. 
Let $\Phi\in \bD_0(D)$ be such that 
$\Phi \geq 1$ on an open neighbourhood of $P_\delta(A)$,
 and let $\Psi=(\Phi \wedge 1) \vee 0$. As before, we have that $\cE(\Psi)\leq \cE(\Phi)$ and that $\Psi = 1$ on an open neighbourhood of $A$. For every $v\in A$ we have that
 \[
\mathsf{R}[\Psi](v) = \frac{1}{\pi \delta_0^2 r(v)^2} \int_{P_{\delta_0}(v)} \Psi(z) \dif z
\geq \frac{1}{\pi \delta_0^2 r(v)^2} \int_{P_{\delta_0}(v)} \mathbf{1}\left[z\in P_\delta(v)\right] \dif z
= \frac{\delta^2}{\delta_0^2}.
 \]
 Thus, by \cref{lem:d0tod0easy}, the function $\delta_0^2\mathsf{R}[\Psi]/\delta^2 \in \bD_0(M)$ is at least $1$ on $A$, and so, by Dirichlet's principle and \cref{lem:AandRenergy},
 \[
\Cap(A) \leq \cE\left(\frac{\delta_0^2}{\delta^2}\mathsf{R}[\Psi]\right) \preceq \delta^{-4} \cE(\mathsf{R}[\Psi])\preceq \delta^{-4}\cE(\Psi) \leq \delta^{-4} \cE(\Phi).
 \] 
 The claimed lower bound follows by taking the infimum over $\Phi$.
\end{proof}

There is one more lemma to prove before we prove \cref{thm:isomorphismgeneral}.

\begin{lemma}\hspace{1cm}
\label{lem:RandAsendD0toD0}
\begin{enumerate}
	\item If $\phi \in \bD(M)$, then $\phi-\mathsf{R}[\mathsf{A}[\phi]]\in \bD_0(M)$.
	\item
If $\phi\in \bD(M)$, then $\mathsf{A}[\phi] \in \bD_0(D)$ if and only if $\phi\in \bD_0(M)$. 
\item
If $\Phi\in \bD(D)$, then $\mathsf{R}[\Phi] \in \bD_0(M)$ if and only if $\Phi\in \bD_0(D)$.
\end{enumerate}
\end{lemma}

\begin{proof}[Proof of \cref{lem:RandAsendD0toD0}]
We begin with item $1$. Observe that, by the definitions of $\mathsf{R}$ and $\mathsf{A}$, we have that
\[\big|\phi(v)-\mathsf{R}[\mathsf{A}[\phi]](v)\big| \leq \sup \left\{|\phi(v)-\phi(u)| : u \text{ shares a face with $v$}\right\}\]
for every vertex $v\in V$. 
It follows by a straightforward argument with the Cauchy-Schwarz inequality, similar to that used in the proof of \cref{lem:AandRenergy}, that 
\[\sum_{v\in V} \big|\phi(v)-\mathsf{R}[\mathsf{A}[\phi]](v)\big|^2 \preceq  \cE(\phi), \]
and hence that, for each $\eps>0$, 
\[
\Cap\Big(\big\{ v \in V : \big|\phi(v)-\mathsf{R}[\mathsf{A}[f]](v)\big| \geq \eps \big\}\Big) \preceq \Big|\big\{ v \in V : \big|\phi(v)-\mathsf{R}[\mathsf{A}[\phi]](v)\big| \geq \eps \big\}\Big|
\preceq  \cE(\phi) \eps^{-2}.
\]
The right hand side is finite for every $\eps>0$, and so we conclude by applying \cref{lem:D0char}.

We now turn to items $2$ and $3$. 
The `if' parts of the statements are covered by \cref{lem:d0tod0easy}; 
It remains to prove only the `only if' parts of the statements.
 We begin with item 2.
 Let $\phi\in \bD(M)$ be such that $\mathsf{A}[\phi] \in \bD_0(D)$ and let $\eps>0$. It follows from \cref{lem:continuity} that there exists a constant $\delta=\delta(\eps,\cE(\phi),\bM(M))$ such that
\[
\{v\in V: |\phi(v)| \geq \eps \} \subseteq \left\{v\in V: |\mathsf{A}[\phi](z)| \geq \frac{\eps}{2} \text{ for all $z\in P_\delta(v)$} \right\},
\] 
and it follows from \cref{lem:CapComparison} that 
there exists a constant $C=C(\eps,\cE(\phi),\bM(M))$ such that
\[
\Cap\left(\left\{v\in V: |\phi(v)|\geq \eps\right\}\right) \leq C\, \Cap\left(\left\{z\in D : |\mathsf{A}[\phi](v)|\geq \frac{\eps}{2} \right\}\right).
\]
Here we have used the fact that if $A \subseteq B$ then $\Cap(A)\leq \Cap(B)$, which is an immediate consequence of the Dirichlet principle. 
 \cref{lem:D0char} and the assumption that $\mathsf{A}[\phi]\in \bD_0(D)$ implies that the right hand side is finite, so that the left hand side is finite also. Since $\eps>0$ was arbitrary, applying \cref{lem:D0char} a second time shows that $\phi \in \bD_0(M)$ as claimed.

It remains to prove item 3. We begin by proving that for every $H\in \HD(D)$ and $\eps>0$ there exists a compact set $K \subset D$ such that
\begin{equation}
\label{eq:movedclaim}
\Cap\bigl(\{z\in D : |H(z)| \geq \eps \}\bigr)\preceq \Cap(K) + \Cap\left[ \left\{ v \in V,\, |H\circ z(u)| \geq \eps/4 \right\}\right].
\end{equation}
For each $v\in V$, define $\mathrm{Fl}(v)$ to be the union of the disc $P(v)$ with all of the discs $P^\dagger(f)$ where $f$ is a face of $M$ incident to $v$, and let $N(v)$ be the set of all vertices of $M$ that share a face with $v$. 
Let $H\in \HD(D)$ and let $\eps>0$. 
Observe that
\begin{align*}
\{z\in D : |H(z)|\geq \eps \} &\subseteq \bigcup \left\{ P(v) : v\in V,\, \sup\{|H(z)| : z \in P(v) \} \geq \eps \right\} \\ &\hspace{3.8cm}\cup\, \bigcup \left\{ P^\dagger(f) : f\in F,\, \sup\{|H(z)| : z \in P^\dagger(f) \}  \geq \eps\right\} 
\\
&\subseteq 
\left\{ \mathrm{Fl}(v) : v\in V,\, \sup\left\{|H(z)| : z \in P(v) \right\} \geq \eps \right\},
\end{align*}
where the second inclusion follows from the maximum principle.
 Define the sets
$A_{\eps,1} = \{v\in V : |H\circ z(v)| \geq \eps/2 \}$ 
and
\[
A_{\eps,2} =
\left\{v\in V : \sup\bigl\{ |H(z)|: z \in P(v) \bigr\} \geq \eps\right\}.
\]
Clearly $A_{\eps,1} \subseteq A_{\eps,2}$. We claim that $A_{\eps,2} \setminus A_{\eps,1}$ is finite. Indeed, suppose for contradiction that $A_{\eps,2} \setminus A_{\eps,1}$ is infinite. It follows from the Ring Lemma that there exists a constant $C>1$ such that $B(z(v),Cr(v)) \subseteq D$ for every $v\in V$, and since the point set $\{z(v):v \in V\}$ is locally finite in $D$, we can find an infinite set $A_{\eps,3} \subseteq A_{\eps,2} \setminus A_{\eps,1}$ such that the balls $B(z(v),Cr(v))$ and $B(z(u),Cr(u))$ are disjoint whenever $u,v\in A_{\eps,3}$ are distinct. Applying item 2 of \cref{lem:continuity} we obtain that
\[
\cE(H) \geq \sum_{v\in A_{\eps,3}} \int_{B(z(v),Cr(v))} \|\nabla H(z)\|^2 \dif z \succeq \sum_{v\in A_{\eps,3}} \eps^2 = \infty,
\]
contradicting the assumption that $H\in \HD(D)$. It follows that if  $H\in \HD(D)$ then
\[
\{z\in D : |H(z)| \geq \eps \}
\subseteq K' \cup \bigcup \left\{\mathrm{Fl}(v) : v \in V,\, |H\circ z(v)| \geq \eps/2 \right\}
\]
where $K' \subset D$ is compact. 
Now, since $H \circ z \in \bD(M)$ by \cref{lem:AandRenergy}, it follows by similar reasoning to above that $\{v\in V : |H \circ z(u)|\geq \eps/2$ for some $u\in N(v)\} \setminus
\{v\in V : H\circ z(u) \geq \eps/4$ for every $u\in \{v\}\cup N(v)\}$ is finite, and it follows that there exists a compact set $K \subset D$ such that
\begin{multline*}
\{z\in D : |H(z)| \geq \eps \}
\subseteq K \cup \bigcup \left\{\mathrm{Fl}(v) : v \in V,\, |H\circ z(u)| \geq \eps/4 \text{ for every } u\in \{v\}\cup N(v) \right\}
\end{multline*}
Now suppose that $\psi \in \bD_0$ is such that $\psi \geq 1$ on the set $\{ v \in V : |H \circ z(v)| \geq \eps/4\}$. Then we clearly have that $\mathsf{A}[\psi] \geq 1$ on the set 
$\bigcup \left\{\mathrm{Fl}(v) : v \in V,\, |H\circ z(u)| \geq \eps/4 \text{ for every $u\in \{v\}\cup N(v)$}\right\}$, and optimizing over $\psi$ it follows that
\begin{multline*}
\Cap\bigl(\{z\in D : |H(z)| \geq \eps \}\bigr)\\
\leq \Cap(K') + \Cap\left[ \bigcup \left\{\mathrm{Fl}(v) : v \in V,\, |H\circ z(u)| \geq \eps/4 \text{ for every $u\in \{v\}\cup N(v)$}\right\}\right]
\\\preceq \Cap(K) + \Cap\left[ \left\{ v \in V,\, |H\circ z(u)| \geq \eps/4 \right\}\right]
\end{multline*}
as claimed. 

Now let $\Phi= \Phi_0 + \Phi_{\mathrm{HD}} \in \bD(D)$ and suppose that $R[\Phi] \in \bD_0(M)$. We have by \cref{lem:d0tod0easy} that $R[\Phi_0] \in \bD_0(M)$, and it follows that $R[\Phi_\mathrm{HD}] = \Phi_\mathrm{HD}\circ z = R[\Phi]-R[\Phi_0] \in \bD_0(M)$ also. Let $\eps>0$. Then we have by \eqref{eq:movedclaim} and \cref{lem:D0char} that there exists a compact subset $K$ of $D$ such that 
\begin{equation*}
\Cap\bigl(\{z\in D : |\Phi_\mathrm{HD}(z)| \geq \eps \}\bigr)
\leq \Cap(K) + \Cap\left[\left\{ v \in V,\, |\Phi_\mathrm{HD}\circ z(v)| \geq \eps/4 \right\}\right] <\infty
\end{equation*}
where we have used the fact that compact subsets of transient domains have finite capacity. 
Since $\eps>0$ was arbitrary it follows from \cref{lem:D0char} that $\Phi_\mathrm{HD}\in \bD_0(D)$, and hence that $\Phi_\mathrm{HD}\equiv 0$ by uniqueness of the Royden decomposition. Thus, $\Phi\in \bD_0(D)$ as claimed. \qedhere

\end{proof}



We are now ready to prove \cref{thm:isomorphismgeneral}. 

\begin{proof}[Proof of \cref{thm:isomorphismgeneral}.]
As discussed after the proof of \cref{lem:AandRenergy}, \cref{lem:AandRenergy} implies that $\Disc$ and $\Cont$ are both bounded. 
Thus, it suffices to prove the following:
\begin{enumerate}
  \item For each $H \in \HD(D)$, $h=\Disc [H]=(\mathsf{R}[H])_\mathrm{HD}$ is the unique element of $\HD(M)$ such that $\mathsf{R}[H] - h \in \bD_0(M)$.
  \item For each $h \in \HD(M)$, $H=\Cont [h]$ is the unique element of $\HD(D)$ such that $h - \mathsf{R}[H] \in \bD_0(M)$. 
  \item $h=\Disc[\Cont[h]]$ and $H=\Cont[\Disc[H]]$ for every $h\in \HD(M)$ and $H\in \HD(D)$ respectively.
\end{enumerate}
Each of these items has a highly elementary but slightly tricky proof. Let $\mathsf{P}_{\bD_0(M)},$ $\mathsf{P}_{\HD(M)}$, $\mathsf{P}_{\bD_0(D)},$ and $\mathsf{P}_{\HD(D)}$  be the projections associated to the Royden decompositions of $\bD(M)$ and $\bD(D)$ respectively. 
\begin{enumerate}
  \item This follows immediately from the uniqueness of the Royden decomposition (i.e., the fact that $\bD(D)=\bD_0(D)\oplus \HD(D)$).
  \item
We first wish to prove that $h-\mathsf{R}\Cont[h] = h - \mathsf{R} \mathsf{P}_{\HD(D)} \mathsf{A} h \in \bD_0(M)$ for every $h\in \bD(M)$. To see this, note that
$h-\mathsf{R}\mathsf{P}_{\HD(D)} \mathsf{A} h  = \left[h- \mathsf{R}\mathsf{A}h\right] + \mathsf{R} \mathsf{P}_{\bD_0(D)} \mathsf{A} h$.
 Since $h-\mathsf{R}\mathsf{A}h \in \bD_0(M)$ by item 1 of \cref{lem:RandAsendD0toD0} and $\mathsf{R} \mathsf{P}_{\bD_0(D)} \mathsf{A} h \in \bD_0(M)$ by \cref{lem:d0tod0easy}, we deduce that $h-\mathsf{R}\Cont[h] \in \bD_0(M)$ as claimed. 

We now prove uniqueness. Suppose that $H\in \HD(D)$ is such that $h-\mathsf{R}[H]$ is in $\bD_0(M)$. Then we must have that $\mathsf{R}\left[\Cont[h]-H\right] = (h-\mathsf{R}[H]) - (h-\mathsf{R}[\Cont[h]])$ is in $\bD_0(M)$ also, and it follows from \cref{lem:RandAsendD0toD0} (more specifically the `only if' implication of item 3 of that lemma) that $\Cont[h]-H \in \bD_0(D)$. But since $\Cont[h]-H \in \HD(D)$ we deduce that $H=\Cont[h]$ as claimed.

\item We first prove that $h=\Disc[\Cont[h]]$ for every $h\in \HD(M)$. We have that $h-\Disc[\Cont[h]] =h- \mathsf{R} \Cont [h] + \mathsf{P}_{\bD_0(M)} \mathsf{R} \Cont [h]$, and since, by item 2, $h- \mathsf{R} \Cont [h]$ and $\mathsf{P}_{\bD_0} \mathsf{R} \Cont [h]$ are both in $\bD_0(M)$, it follows that $h-\Disc[\Cont[h]] \in \bD_0(M)$ and hence that $h-\Disc[\Cont[h]]=0$ as claimed. 

It remains to prove that $H=\Cont[\Disc[H]]$ for every $H\in\HD(D)$. By item 2 we have that 
$\Disc[H] - \mathsf{R} \Cont[\Disc[H]]  \in \bD_0(M)$, and hence that
\[\mathsf{R}\bigl[H -  \Cont[\Disc[H]] \bigr]= \mathsf{P}_{\bD_0(M)}\mathsf{R} [H] + \Disc[H] - \mathsf{R} \Cont[\Disc[H]] \in \bD_0(M)\]also. It follows by \cref{lem:RandAsendD0toD0} that $H -  \Cont[\Disc[H]]\in \bD_0(D)$ and hence that $H -  \Cont[\Disc[H]]=0$ as claimed. \qedhere
\end{enumerate}
\end{proof}

\subsection{Asymptotic equality in the uniformly transient case}

We now prove the following proposition, which, together with Proposition \ref{lem:D0char}, allows us to deduce \cref{thm:isomorphismdisc} from \cref{thm:isomorphismgeneral}.

\begin{prop}
\label{prop:uniformlytransient}
Let $M$ be a transient weighted polyhedral planar map with bounded codegrees and bounded local geometry, let $(P,P^\dagger)$ be a double circle packing of $M$ in a domain $D \subset \C$, and let $z:V\to D$ be the function sending each circle to the centre of its corresponding disc. 
Let $h$ and $H$ be bounded harmonic functions on $M$ and $D$ respectively. If $D$ is uniformly transient, then $h$ and $H\circ z$ are asymptotically equal if and only if they are quasi-asymptotically equal. 
\end{prop}

The proof of this proposition applies the elliptic Harnack inequality, which we now discuss. 
For each $z\in \C$ and $r>0$, let $B(z,r)$ denote the Euclidean ball of radius $r$ around $z$.
Recall the classical elliptic Harnack inequality for the plane, which states that for every $z_0\in \C$, every non-negative harmonic function $h: B(z_0,r) \to \R$, and every $z\in B(z_0,r)$, we have that
\begin{equation}
\label{eq:classicerEHI}
\frac{r-|z-z_0|}{r+|z-z_0|} h(z_0)
\leq
h(z)\leq \frac{r+|z-z_0|}{r-|z-z_0|} h(z_0).
\end{equation}
An immediate consequence of this inequality is that
\begin{equation}
\label{eq:classicerEHI2}
|h(z)-h(z_0)|\leq \frac{2|z-z_0|}{r-|z-z_0|} h(z_0)
\end{equation}
under the same assumptions. If $h:B(z_0,r)\to \R$ is a harmonic function that is not necessarily non-negative, we can apply this inequality to the normalized function $h-\inf_{z\in B(z_0,r)} h(z)$ to obtain that
\begin{multline}
\label{eq:classicEHI}
|h(z)-h(z_0)| \leq 
\frac{2|z-z_0|}{r-|z-z_0|} \bigl(h(z_0)-\inf_{z' \in B(z_0,r)} h(z')\bigr) 
\\
\leq 
\frac{2 |z-z_0|}{r-|z-z_0|}\sup\big\{|h(z_1)-h(z_2)| : z_1,z_2 \in B(z_0,r) \big\}.
\end{multline}

Angel, Barlow, Gurel-Gurevich, and Nachmias \cite{ABGN14} established a version of the elliptic Harnack inequality that holds for double circle packings with respect to the Euclidean metric.
 The version of the theorem that we state here follows from that stated in \cite{ABGN14} by a simple rearrangement and iteration argument, below.

\begin{theorem}[Elliptic Harnack Inequality]
Let $M$ be a transient weighted polyhedral planar map with bounded codegrees and bounded local geometry, let $(P,P^\dagger)$ be a double circle packing of $M$ in a domain $D$. Then for each $\alpha<1$ there exist positive constants $\beta=\beta (\bM)$ and $C=C(\bM)$ such that 
\begin{equation}
\label{eq:DiscreteEHI}
 |h(u)-h(v)| \leq C \left(\frac{|z(u)-z(v)|}{r}\right)^\beta \sup\big\{
 |h(w_1)-h(w_2)| : z(w_1),z(w_2) \in B(z,r) \big\}
\end{equation}
for every harmonic function $h$ on $V$, every $v\in V$, every $r \leq d(z(v),\partial D)$, and every $u\in V$ with $z(u) \in B(z(v),\alpha r)$.
\end{theorem}

\begin{proof}
 Let $X$ be the union of the straight lines between the centres of circles in $P$. The Ring Lemma implies that the path metric on $X$ is comparable to the subspace metric on $X$ \cite[Proposition 2.5]{ABGN14}. Given a function $\phi$ on the vertex set of $M$, we extend $\phi$ to $X$ by linear interpolation along each edge.
The version of the elliptic Harnack inequality stated in \cite[Theorem 5.4]{ABGN14}
 implies that for each $A>1$, there exists a constant $C=C(A,\bM)>1$ such that for every $x \in X$ with $d(x,\partial D) \geq Ar$, and every harmonic function $h$ on $M$ such that the extension of $h$ to $X$ is positive on $B(x,Ar)$, we have that
 \begin{equation}
 \label{eq:ABGNEHI}
\sup_{y \in X \cap B(x,r)} h(y) \leq C \inf_{y\in X \cap B(x,r)} h(y).
 \end{equation}
Now suppose that $h$ is a harmonic function on $M$ that is not necessary positive. Write $B(r)=X\cap B(x,r)$. Applying this inequality to the normalized function $h(y)-\inf_{z \in B(Ar)}h(z)$, we deduce that
\[
\sup_{y \in B(r)} h(y)-\inf_{y \in B(Ar)} h(y) \leq C \left[ \inf_{y \in B(r)} h(y)-\inf_{y \in B(Ar)} h(y) \right].
\]
Adding $(C-1) \sup_{y \in B(r)} h(y) + \inf_{y\in B(Ar)} h(y) - C \inf_{y\in B(r)} h(y)$ to both sides of this inequality, we obtain that
\begin{align*}
C\left[\sup_{y \in B(r)} h(y) - \inf_{y \in B(r)} h(y)\right] &\leq (C-1) \sup_{y\in B(r)} h(y) - (C-1) \inf_{y \in B(Ar)} h(y)\\
&\leq (C-1)\left[ \sup_{y \in B(Ar)} h(y) - \inf_{y \in B(Ar)} h(y)\right].
\end{align*}
By applying this inequality for different values of $r$ we obtain that
\[
\sup_{y \in B(A^{-n}r)}h(y) - \inf_{y \in B(A^{-n}r)}h(y) \leq \left(\frac{C-1}{C}\right) \left[ \sup_{y \in B(A^{-n+1}r)} h(y) - \inf_{y \in B(A^{-n+1}r)} h(y) \right]
\]
for  every $n\geq 1$, every harmonic function $h$ on $M$, every $r>0$, every $n\geq 1$, and every $x\in X$ such that $d(x,\partial D) \geq r$. It follows by induction that
\[
\sup_{y \in B(A^{-n}r)}h(y) - \inf_{y \in B(A^{-n}r)}h(y) \leq \left(\frac{C-1}{C}\right)^n \left[ \sup_{y \in B(r)} h(y) - \inf_{y \in B(r)} h(y) \right]
\]
for  every harmonic function $h$ on $M$, every $r>0$, every $n\geq 1$, and every $x\in X$ such that $d(x,\partial D) \geq r$. This is easily seen to imply the claimed inequality
\end{proof}

The following lemma is presumably well-known to experts, but we were not able to find a reference.

\begin{lemma}
\label{lem:muchmoredetail}
Let $G$ be a transient network and suppose that $A$ is a set of vertices for which there exists $\eps>0$ and infinitely many disjoint sets $A_1,A_2,
\ldots \subseteq A$ such that $\Cap(A_i)\geq \eps$ for every $i\geq 1$. Then $\Cap(A)=\infty$.
\end{lemma}

\begin{proof}
First note that if $A$ has finite capacity then we must have that simple random walk on $G$ visits $A$ at most finitely often almost surely. Indeed, if $\Cap(A)<\infty$ then there exists $\psi \in \bD_0(G)$ with $\psi|_A\geq 1$, and it follows from \eqref{eq:ALPlimitidentification} that if $X$ is a random walk then $\psi(X_n)\to 0$ a.s.\ and hence that $X$ visits $A$ at most finitely often a.s.  Thus, it suffices to consider the case that the simple random walk visits $A$ at most finitely often almost surely. 

For each $i\geq 1$, there exists a finite set $A_i' \subseteq A_i$ such that $\Cap(A'_i) \geq \Cap(A_i)/2 \geq \eps/2$. We construct a subsequence 
$i_1,i_2,\ldots$ as follows. Let $i_1=1$. Since random walk visits $A$ at most finitely often almost surely, it follows that, 
given $i_1,\ldots,i_m$, there exists $j$ such that 
\[\sum_{\ell=1}^m \sum_{v \in A_{i_\ell}'} c(v) \bP_v\Bigl(\text{hit }\bigcup_{i\geq j} A'_i\Bigr) \leq \eps/8\]
Set $i_m$ to be the minimal such $j$; this gives a recursive procedure to define the entire sequence $i_1,i_2,\ldots$.
By the Dirichlet principle we have that $\Cap(A) \geq \Cap\Bigl(\bigcup_{\ell=1}^m A_{i_\ell}'\Bigr)$ for each $m\geq 1$, and so it suffices to prove that
\begin{equation}
\label{eq:capacity_claim}
\Cap\Bigl(\bigcup_{\ell=1}^m A_{i_\ell}'\Bigr) \geq \frac{\eps m}{4}
\end{equation}
for every $m\geq 1$. 
To see this, we use the elementary bound
\begin{align*}
\Cap\Bigl(\bigcup_{\ell=1}^m A_{i_\ell}'\Bigr)  &= \sum_{\ell=1}^m \sum_{v \in A_{i_\ell}'} c(v)\mathbf{P}_v\Bigl(\text{do not return to $\bigcup_{\ell=1}^m A_{i_\ell}'$} \Bigr)\\
&\geq \sum_{\ell=1}^m \sum_{v\in  A_{i_\ell}'} c(v)\mathbf{P}_v\Bigl(\text{do not return to  $A_{i_\ell}'$} \Bigr) - \sum_{\ell=1}^m \sum_{v\in  A_{i_\ell}'} c(v) \mathbf{P}_v\Bigl(\text{hit $\bigcup_{k\geq \ell+1 }A'_{i_k}$} \Bigr) 
\\ &\hspace{4cm}-
\sum_{\ell=1}^m \sum_{v\in  A_{i_\ell}'} \sum_{r=1}^{\ell-1} \sum_{u \in A_{i_r}'}
c(v) \mathbf{P}_v\Bigl(\text{hit $u$, don't return to $\bigcup_{k\geq \ell} A'_{i_k}$} \Bigr),
\end{align*}
from which the bound
\[
\Cap\Bigl(\bigcup_{\ell=1}^m A_{i_\ell}'\Bigr)
\geq \frac{\eps m}{2} - \frac{\eps m}{8}  - \sum_{\ell=1}^m \sum_{v\in  A_{i_\ell}'} \sum_{r=1}^{l-1} \sum_{u \in A_{i_r}'}
c(v) \mathbf{P}_v\Bigl(\text{hit $u$, don't return to $\bigcup_{k \geq \ell} A'_{i_k}$} \Bigr)
\]
follows immediately. To control the final term, we reverse time to get that
\begin{multline*}
\Cap\Bigl(\bigcup_{\ell=1}^m A_{i_\ell}'\Bigr)
\geq  \frac{3\eps m}{8}  -  \sum_{r=1}^m \sum_{u \in A_{i_r}'} \sum_{\ell=r+1}^m \sum_{v\in  A_{i_\ell}'}
c(u) \mathbf{P}_u\Bigl(\text{hit $\bigcup_{k\geq \ell} A'_{i_k}$ for first time at $v$} \Bigr)\\
\geq \frac{3\eps m}{8}  - \sum_{r=1}^m \sum_{u \in A_{i_r}'} \sum_{\ell=r+1}^m
c(u) \mathbf{P}_u\Bigl(\text{hit $\bigcup_{k\geq \ell} A'_{i_k}$} \Bigr)\\
\geq \frac{3\eps m}{8}  - m \sum_{r=1}^m \sum_{u \in A_{i_r}'} 
c(u) \mathbf{P}_u\Bigl(\text{hit $\bigcup_{k \geq r+1} A'_{i_k}$} \Bigr)
\geq \frac{\eps m}{4}
\end{multline*}
as claimed. The claim that $A$ has infinite capacity now follows immediately from \eqref{eq:capacity_claim}.
\end{proof}

\begin{proof}[Proof of \cref{prop:uniformlytransient}]
Asymptotic equality clearly implies quasi-asymptotic equality. Suppose that $h$ and $H\circ z$ are not asymptotically equal, so that there exists $\eps>0$ such that the set $A_{\eps}=\{ v\in V : |h(v)-H\circ z(v)| \geq \eps\}$ is infinite. Since $h$ and $H$ are bounded, it follows from the elliptic Harnack inequalities \eqref{eq:classicEHI} and \eqref{eq:DiscreteEHI} that there exists $\delta>0$ such that  \[ 
\bigcup_{v\in A_\eps}\Big\{u \in V : z(u) \in B\Big(z(v),\delta d\big(z(v),\partial D\big)\Big)\Big\} \subseteq A_{\eps/2}.\] 
Since $D$ is uniformly transient, \cref{lem:CapComparison} implies that the sets 
\[
\Big\{z \in D : z \in B\Big(z(v),\delta d\big(z(v),\partial D\big)\Big)\Big\}
\]
have capacity bounded below by some positive constant, and a simple variation on the proof of \cref{lem:CapComparison} yields that the sets 
\[\Bigl\{u \in V : z(u) \in B\Big(z(v),\delta d\big(z(v),\partial D\big)\Big)\Big\}\]
also have capacity bounded below by a positive constant. Since there must exist infinitely many disjoint sets of this form, we can apply \cref{lem:muchmoredetail} to deduce that $\Cap(A_{\eps/2})=\infty$. It follows that $h$ and $H \circ z$ are not quasi-asymptotically equal, concluding the proof.
\end{proof}

\subsection*{Acknowledgments}
The author was supported by a Microsoft Research PhD Fellowship. We thank the anonymous referees for their comments and corrections.

 \setstretch{1}
 \small{
  \bibliographystyle{abbrv}
  \bibliography{unimodularthesis.bib}

\begin{thebibliography}{10}

\bibitem{Ahar97}
D.~Aharonov.
\newblock The sharp constant in the ring lemma.
\newblock {\em Complex Variables Theory Appl.}, 33(1-4):27--31, 1997.

\bibitem{AnLyPe99}
A.~Ancona, R.~Lyons, and Y.~Peres.
\newblock Crossing estimates and convergence of {D}irichlet functions along
  random walk and diffusion paths.
\newblock {\em Ann. Probab.}, 27(2):970--989, 1999.

\bibitem{andreev1970convex}
E.~M. Andreev.
\newblock Convex polyhedra in {L}oba\v cevski\u\i \ spaces.
\newblock {\em Mat. Sb. (N.S.)}, 81 (123):445--478, 1970.

\bibitem{ABGN14}
O.~Angel, M.~T. Barlow, O.~Gurel-Gurevich, and A.~Nachmias.
\newblock Boundaries of planar graphs, via circle packings.
\newblock {\em Ann. Probab.}, 44(3):1956--1984, 2016.

\bibitem{unimodular2}
O.~Angel, T.~Hutchcroft, A.~Nachmias, and G.~Ray.
\newblock Hyperbolic and {P}arabolic {U}nimodular {R}andom {M}aps.
\newblock {\em Geom. Funct. Anal.}, 28(4):879--942, 2018.

\bibitem{MR3755822}
J.~Ashe, E.~Crane, and K.~Stephenson.
\newblock Circle packing with generalized branching.
\newblock {\em J. Anal.}, 24(2):251--276, 2016.

\bibitem{BS96a}
I.~Benjamini and O.~Schramm.
\newblock Harmonic functions on planar and almost planar graphs and manifolds,
  via circle packings.
\newblock {\em Invent. Math.}, 126(3):565--587, 1996.

\bibitem{doob1962boundary}
J.~Doob.
\newblock Boundary properties of functions with finite dirichlet integrals.
\newblock In {\em Annales de l'institut Fourier}, volume~12, pages 573--621,
  1962.

\bibitem{MR0109961}
J.~L. Doob.
\newblock Conditional {B}rownian motion and the boundary limits of harmonic
  functions.
\newblock {\em Bull. Soc. Math. France}, 85:431--458, 1957.

\bibitem{MR0173783}
J.~L. Doob.
\newblock Boundary properties for functions with finite {{D}}irichlet
  integrals.
\newblock {\em Ann. Inst. Fourier (Grenoble)}, 12:573--621, 1962.

\bibitem{MR1501590}
J.~Douglas.
\newblock Solution of the problem of {P}lateau.
\newblock {\em Trans. Amer. Math. Soc.}, 33(1):263--321, 1931.

\bibitem{MR3185375}
O.~El-Fallah, K.~Kellay, J.~Mashreghi, and T.~Ransford.
\newblock {\em A primer on the {D}irichlet space}, volume 203 of {\em Cambridge
  Tracts in Mathematics}.
\newblock Cambridge University Press, Cambridge, 2014.

\bibitem{Gab05}
D.~Gaboriau.
\newblock Invariant percolation and harmonic {D}irichlet functions.
\newblock {\em Geom. Funct. Anal.}, 15(5):1004--1051, 2005.

\bibitem{G13}
A.~Georgakopoulos.
\newblock The boundary of a square tiling of a graph coincides with the
  {P}oisson boundary.
\newblock {\em Invent. Math.}, 203(3):773--821, 2016.

\bibitem{georgakopoulos2015group}
A.~Georgakopoulos.
\newblock Group walk random graphs.
\newblock {\em Groups, Graphs, and Random Walks}, (436), 2016.

\bibitem{GGNS15}
O.~Gurel-Gurevich, A.~Nachmias, and J.~Souto.
\newblock Recurrence of multiply-ended planar triangulations.
\newblock {\em Electron. Commun. Probab.}, 22:Paper No. 5, 6, 2017.

\bibitem{Hansen}
L.~J. Hansen.
\newblock On the {R}odin and {S}ullivan ring lemma.
\newblock {\em Complex Variables Theory Appl.}, 10(1):23--30, 1988.

\bibitem{he1999rigidity}
Z.-X. He.
\newblock Rigidity of infinite disk patterns.
\newblock {\em Annals of Mathematics}, 149:1--33, 1999.

\bibitem{HS93}
Z.-X. He and O.~Schramm.
\newblock Fixed points, {K}oebe uniformization and circle packings.
\newblock {\em Ann. of Math. (2)}, 137(2):369--406, 1993.

\bibitem{HeSc}
Z.-X. He and O.~Schramm.
\newblock Hyperbolic and parabolic packings.
\newblock {\em Discrete Comput. Geom.}, 14(2):123--149, 1995.

\bibitem{holopainen1997p}
I.~Holopainen and P.~M. Soardi.
\newblock p-harmonic functions on graphs and manifolds.
\newblock {\em manuscripta mathematica}, 94(1):95--110, 1997.

\bibitem{HutNach15b}
T.~Hutchcroft and A.~Nachmias.
\newblock Uniform spanning forests of planar graphs.
\newblock http://arxiv.org/abs/1603.07320.

\bibitem{hutchcroft2015boundaries}
T.~Hutchcroft and Y.~Peres.
\newblock Boundaries of planar graphs: a unified approach.
\newblock {\em Electron. J. Probab.}, 22:Paper No. 100, 20, 2017.

\bibitem{K36}
P.~Koebe.
\newblock {\em Kontaktprobleme der konformen Abbildung}.
\newblock Hirzel, 1936.

\bibitem{LZ}
S.~K. Lando and A.~K. Zvonkin.
\newblock {\em Graphs on surfaces and their applications}, volume 141 of {\em
  Encyclopaedia of Mathematical Sciences}.
\newblock Springer-Verlag, Berlin, 2004.
\newblock With an appendix by Don B. Zagier, Low-Dimensional Topology, II.

\bibitem{lee1999rough}
Y.~H. Lee.
\newblock Rough isometry and dirichlet finite harmonic functions on riemannian
  manifolds.
\newblock {\em manuscripta mathematica}, 99(3):311--328, 1999.

\bibitem{LP:book}
R.~Lyons and Y.~Peres.
\newblock {\em Probability on Trees and Networks}.
\newblock Cambridge University Press, New York, 2016.

\bibitem{marden1990thurston}
A.~Marden and B.~Rodin.
\newblock On thurston's formulation and proof of andreev's theorem.
\newblock In {\em Computational methods and function theory}, pages 103--115.
  Springer, 1990.

\bibitem{miermont2014aspects}
G.~Miermont.
\newblock Aspects of random maps.
\newblock {\em Lecture notes of the 2014 Saint-Flour Probability Summer School,
  preliminary notes available at the author's webpage}, 2014.

\bibitem{MR672838}
A.~Nagel, W.~Rudin, and J.~H. Shapiro.
\newblock Tangential boundary behavior of functions in {D}irichlet-type spaces.
\newblock {\em Ann. of Math. (2)}, 116(2):331--360, 1982.

\bibitem{RS87}
B.~Rodin and D.~Sullivan.
\newblock The convergence of circle packings to the {R}iemann mapping.
\newblock {\em J. Differential Geom.}, 26(2):349--360, 1987.

\bibitem{Rohde11}
S.~Rohde.
\newblock {Oded Schramm: From Circle Packing to SLE}.
\newblock {\em Ann. Probab.}, 39:1621--1667, 2011.

\bibitem{MR0049396}
H.~L. Royden.
\newblock Harmonic functions on open {R}iemann surfaces.
\newblock {\em Trans. Amer. Math. Soc.}, 73:40--94, 1952.

\bibitem{Schramm91}
O.~Schramm.
\newblock Rigidity of infinite (circle) packings.
\newblock {\em J. Amer. Math. Soc.}, 4(1):127--149, 1991.

\bibitem{MR0350876}
M.~L. Silverstein.
\newblock Classification of stable symmetric {M}arkov chains.
\newblock {\em Indiana Univ. Math. J.}, 24:29--77, 1974.

\bibitem{Soardi93}
P.~M. Soardi.
\newblock Rough isometries and {D}irichlet finite harmonic functions on graphs.
\newblock {\em Proc. Amer. Math. Soc.}, 119(4):1239--1248, 1993.

\bibitem{Soardibook}
P.~M. Soardi.
\newblock {\em Potential theory on infinite networks}, volume 1590 of {\em
  Lecture Notes in Mathematics}.
\newblock Springer-Verlag, Berlin, 1994.

\bibitem{St05}
K.~Stephenson.
\newblock {\em Introduction to circle packing}.
\newblock Cambridge University Press, Cambridge, 2005.
\newblock The theory of discrete analytic functions.

\bibitem{Th78}
W.~P. Thurston.
\newblock The geometry and topology of 3-manifolds.
\newblock {\em Princeton lecture notes.}, 1978-1981.

\bibitem{yamasaki1986ideal}
M.~Yamasaki et~al.
\newblock Ideal boundary limit of discrete dirichlet functions.
\newblock {\em Hiroshima Mathematical Journal}, 16(2):353--360, 1986.

\end{thebibliography}
}

\end{document}